\documentclass[11pt, a4paper]{article}

\usepackage{amsmath}
\usepackage{amsfonts}
\usepackage{amssymb}
\usepackage[francais,english]{babel}

\def\english{\selectlanguage{english}}

\providecommand\mathbb{\bf}
\newcommand\R{{\mathbb R}}
\newcommand\N{{\mathbb N}}

\newtheorem{thm}{Theorem}[section]
\newtheorem{lemma}{Lemma}[section]
\newtheorem{pro}{Proposition}[section]

\newtheorem{remark}{Remark}[section]

\newcounter{Remark}

\renewcommand\theRemark{\arabic{Remark}}

\newcounter{steps}
\newenvironment{proof}[1][]{%
\par\medbreak\setcounter{steps}{0}
{\noindent\bfseries Proof#1. }} {\hfill\fbox{\ }\medbreak}

\newcounter{substeps}[steps]

\newcommand{\intsy}[1]{
\int _0 ^t\!\!\int _{\R ^m} \!\!#1 \;\mathrm{d}y\mathrm{d}s}

\newcommand{\intty}[1]{
\int _0 ^{+\infty}\!\!\!\!\int _{\R ^m} \!\!#1 \;\mathrm{d}y\mathrm{d}t}

\newcommand{\inty}[1]{
\int _{\R^m}\!\!#1 \;\mathrm{d}y}

\newcommand{\uin}[0]{
u_{\mathrm{in}}}

\newcommand{\vin}[0]{
v_{\mathrm{in}}}

\newcommand{\win}[0]{
w_{\mathrm{in}}}

\newcommand{\dom}[0]{
\mathrm{dom}}

\newcommand{\cdny}[0]{
\cdot \nabla _y }

\newcommand{\dyy}[0]{
\partial _y Y (s;y)}

\newcommand{\dyb}[0]{
\partial _y b}

\newcommand{\ysy}[0]{
Y(s;y)}

\newcommand{\eps}[0]{
\varepsilon}

\newcommand{\lime}[0]{
\lim _{\varepsilon \searrow 0 }}

\newcommand{\ue}[0]{
u ^\varepsilon}

\newcommand{\uein}[0]{
u ^\varepsilon _ {\mathrm{in}}}

\newcommand{\uekin}[0]{
u ^{\varepsilon _k} _ {\mathrm{in}}}

\newcommand{\md}[0]{
\mathrm{d}}

\newcommand{\tue}[0]{
\tilde{u} ^\varepsilon}

\newcommand{\uek}[0]{
u ^{\varepsilon _k}}

\newcommand{\limk}[0]{
\lim _{k \to +\infty}}

\newcommand{\divy}[0]{
\mathrm{div}_y}

\newcommand{\loloc}[0]{
L^1_{\mathrm{loc}}(\R^m)}

\newcommand{\ltloc}[0]{
L^2_{\mathrm{loc}}(\R^m)}

\newcommand{\dpri}[0]{
{\cal D}^{\;\prime}(\R^m)}

\newcommand{\aeps}[0]{
a^\varepsilon}

\newcommand{\ran}[0]{
\mathrm{Range\;}}

\newcommand{\lty}[0]{
L^2(\R ^m)}

\newcommand{\loy}[0]{
L^1(\R ^m)}

\newcommand{\liy}[0]{
L^\infty(\R ^m)}

\newcommand{\kerbg}[0]{
\ker ( b \cdot \nabla _y ) }

\newcommand{\bg}[0]{
b \cdot \nabla _y  }

\newcommand{\hoy}[0]{
{H^1(\R^m)}}

\newcommand{\nablar}[0]{
\nabla ^R}

\newcommand{\litlty}[0]{
{L^\infty (\R_+; L^2 (\R^m))}}

\newcommand{\lttlty}[0]{
{L^2 (\R_+; L^2 (\R^m))}}

\newcommand{\ave}[1]{
\left \langle #1 \right \rangle }

\baselinestretch\renewcommand{\baselinestretch}{1.5}

\begin{document}
\english

\title{Strongly anisotropic diffusion problems; asymptotic analysis}

\author{Mihai Bostan
\thanks{Laboratoire d'Analyse, Topologie, Probabilit\'es LATP, Centre de Math\'ematiques et Informatique CMI, UMR CNRS 7353, 39 rue Fr\'ed\'eric Joliot Curie, 13453 Marseille  Cedex 13
France. E-mail : {\tt bostan@cmi.univ-mrs.fr}}
}

\date{ (\today)}

\maketitle

\begin{abstract}
The subject matter of this paper concerns anisotropic diffusion equations: we consider heat equations whose diffusion matrix have disparate eigenvalues. We determine first and second order approximations, we study the well-posedness of them and establish convergence results. The analysis relies on averaging techniques, which have been used previously for studying transport equations whose advection fields have disparate components.

\end{abstract}



\section{Introduction}
\label{Intro}
\indent
Many real life applications lead to highly anisotropic diffusion equations: flows in porous media, quasi-neutral plasmas, microscopic transport in magnetized plasmas \cite{Bra65}, plasma thrusters, image processing \cite{PerMal90}, \cite{Wei98}, thermal properties of crystals \cite{DiaShaYun91}. In this paper we investigate the behavior of the solutions for heat equations whose diffusion becomes very high along some direction. We consider the problem
\begin{equation}
%\left
%\lbrace 
%\begin{array}{l l}
\label{Equ1} \partial _t \ue - \divy ( D(y) \nabla _y \ue ) - \frac{1}{\eps} \divy ( b(y) \otimes b(y) \nabla _y \ue ) = 0,  \;\;(t,y) \in \R_+ \times \R ^m
\end{equation}
\begin{equation}
\label{Equ2} 
\ue (0,y) = \uein (y),  \;\;y  \in \R^m
%\end{array}
%\right.
\end{equation}
where $D(y) \in {\cal M}_m (\R)$ and $b(y) \in \R^m$ are smooth given matrix field and vector field on $\R^m$, respectively. For any two vectors $\xi, \eta$, the notation $\xi \otimes \eta$ stands for the matrix whose entry $(i,j)$ is $\xi _i \eta _j$, and for any two matrix $A, B$ the notation $A:B$ stands for $\mathrm{trace}(^t AB) = A_{ij} B_{ij}$ (using Einstein summation convention).
We assume that at any $y \in \R^m$ the matrix $D(y)$ is symmetric and $D(y) + b (y) \otimes b(y)$ is positive definite
\begin{equation}
\label{Equ3}
^t D (y) = D(y),\;\;\exists \;d >0 \;\;\mbox{such that}\;\;D(y)\xi\cdot\xi + (b(y) \cdot \xi)^2 \geq d \;|\xi|^2,\;\;\xi \in \R ^m,\;\;y \in \R ^m.
\end{equation}
The vector field $b(y)$, to which the anisotropy is aligned, is supposed divergence free {\it i.e.,} $\divy b = 0$. We intend to analyse the behavior of \eqref{Equ1}, \eqref{Equ2} for small $\eps$, let us say $0 < \eps \leq 1$. In that cases $D(y) + \frac{1}{\eps} b(y) \otimes b(y)$ remains positive definite and if $(\uein)_\eps$ remain in a bounded set of $\lty$, then $(\ue)_\eps$ remain in a bounded set of $\litlty{}$ since, for any $t \in \R_+$ we have 
\begin{align*}
\frac{1}{2}\inty{(\ue (t,y))^2} & + d \intsy{|\nabla _y \ue (s,y)|^2}  \leq \frac{1}{2}\inty{(\ue (t,y))^2} \\
& + \intsy{\left \{D(y) + \frac{1}{\eps} b(y) \otimes b(y) \right \} : \nabla _y \ue (s,y) \otimes \nabla _y \ue (s,y)} \\
& = \frac{1}{2}\inty{(\uein (y))^2}. 
\end{align*}
In particular, when $\eps \searrow 0$, $(\ue)_\eps$ converges, at least weakly $\star$ in $\litlty{}$ towards some limit $u \in \litlty{}$. Notice that the explicit methods are not well adapted for the numerical approximation of \eqref{Equ1}, \eqref{Equ2} when $\eps \searrow 0$, since the CFL condition leads to severe time step constraints like
\[
\frac{d}{\eps} \frac{\Delta t}{|\Delta y |^2} \leq \frac{1}{2}
\]
where $\Delta t$ is the time step and $\Delta y $ is the grid spacing. In such cases implicit methods are desirable \cite{BalTilHow08}, \cite{ShaHam10}.

Rather than solving \eqref{Equ1}, \eqref{Equ2} for small $\eps >0$, we concentrate on the limit model satisfied by the limit solution $u = \lime \ue$. We will see that the limit model is still a parabolic problem, decreasing the $\lty$ norm and satisfying the maximum principle. At least formally, the limit solution $u$ is the dominant term of the expansion
\begin{equation}
\label{Equ6} \ue = u + \eps u ^1 + \eps ^2 u ^2 + ...
\end{equation}
Plugging the Ansatz \eqref{Equ6} into \eqref{Equ1} leads to 
\begin{equation}
\label{Equ7} \divy (b \otimes b \nabla _y u ) = 0,\;\;(t,y) \in \R_+ \times \R ^m
\end{equation}
\begin{equation}
\label{Equ8} \partial _t u - \divy (D \nabla _y u ) - \divy ( b \otimes b \nabla _y u^1) = 0,\;\;(t,y) \in \R_+ \times \R ^m
\end{equation}
\[
\vdots
\]
Clearly, the constraint \eqref{Equ7} says that at any time $t \in \R_+$, $b \cdot \nabla _y u = 0$, or equivalently $u(t,\cdot)$ remains constant along the flow of $b$, see \eqref{EquFlow}
\[
u(t, Y(s;y)) = u(t,y),\;\;s \in \R,\;\;y \in \R^m.
\]
The closure for $u$ comes by eliminating $u^1$ in \eqref{Equ8}, combined with the fact that \eqref{Equ7} holds true at any time $t \in \R_+$. The symmetry of the operator $\divy (b \otimes b \nabla _y)$ implies that $\partial _t u - \divy (D \nabla _y u)$ belongs to $(\ker (b \cdot \nabla _y ))^\perp$ and therefore we obtain the weak formulation
\begin{equation}
\label{Equ9}
\frac{\md}{\md t}\inty{u(t,y) \varphi (y)} + \inty{D \nabla _y u (t,y) \cdot \nabla _y \varphi (y) } = 0,\;\;\varphi \in \hoy \cap \kerbg{}.
\end{equation}
The above formulation is not satisfactory, since the choice of test functions is constrained by \eqref{Equ7}; \eqref{Equ9} is useless for numerical simulation. A more convenient situation is to reduce \eqref{Equ9} to another problem, by removing the constraint \eqref{Equ7}. The method we employ here is related to the averaging technique which has been used to handle transport equations with diparate advection fields \cite{BosAsyAna}, \cite{BosTraSin}, \cite{BosGuidCent3D}, \cite{Bos12}
\begin{equation}
\label{Equ10} \partial _t \ue + a(t,y) \cdot \nabla _y \ue + \frac{1}{\eps} b (y) \cdot \nabla _y \ue = 0,\;\;(t,y) \in \R_+ \times \R^m
\end{equation}
\begin{equation}
\label{Equ11} \ue (0,y) = \uein (y),\;\;y \in \R^m.
\end{equation}
Using the same Ansatz \eqref{Equ6} we obtain as before that $b \cdot \nabla _y u (t,\cdot) = 0, t \in \R_+$ and the closure for $u$ writes
\begin{equation}
\label{Equ12}
\mathrm{Proj}_{\kerbg} \{ \partial _t u + a\cdot \nabla _y u \} = 0
\end{equation}
or equivalently
\begin{equation}
\label{Equ14} \frac{\md }{\md t} \inty{u(t,y) \varphi (y) } - \inty{u(t,y) \;a \cdot \nabla _y \varphi } = 0
\end{equation}
for any smooth function satisfying the constraint $b \cdot \nabla _y \varphi = 0$. The method relies on averaging since the projection on $\kerbg$ coincides with the average along the flow of $b$, cf. Proposition \ref{AverageOperator}. As $u$ satisfies the constraint $b \cdot \nabla _y u = 0$, it is easily seen that $\mathrm{Proj}_{\kerbg} \partial _t u = \partial _t u$. A simple case to start with is when the transport operator $a \cdot \nabla _y$ and $b \cdot \nabla _y$ commute {\it i.e.,} $[b \cdot \nabla _y, a \cdot \nabla _y ] = 0$. In this case $a \cdot \nabla _y$ leaves invariant the subspace of the constraints, implying that $\mathrm{Proj}_{\kerbg} \{a \cdot \nabla _y u \} = a \cdot \nabla _y u$. Therefore \eqref{Equ12} reduces to a transport equation and it is easily seen that this equation propagates the constraint, which allows us to remove it. Things happen similarly when the transport operators $a \cdot \nabla _y, b \cdot \nabla _y$ do not commute, but the transport operator of the limit model may change. In \cite{BosTraSin} we prove that there is a transport operator $A \cdot \nabla _y$, commuting with $b \cdot \nabla _y$, such that for any $u \in \kerbg$ we have
\[
\mathrm{Proj}_{\kerbg} \{a \cdot \nabla _y u \} = A \cdot \nabla _y u.
\]
Once we have determined the field $A$, \eqref{Equ12} can be replaced by $\partial _t u + A \cdot \nabla _y u = 0$, which propagates the constraint $b \cdot \nabla _y u (t) = 0$ as well.

Comming back to the formulation \eqref{Equ9}, we are looking for a matrix field $\tilde{D}(y)$ such that $\divy (\tilde{D} \nabla _y)$ commutes with $b \cdot \nabla _y$ and
\[
\mathrm{Proj}_{\kerbg} \{\divy (D(y) \nabla _y u ) \}= \divy (\tilde{D}(y)\nabla _y u),\;\;u \in \kerbg{}.
\]
We will see that, under suitable hypotheses, it is possible to find such a matrix field $\tilde{D}$, and therefore \eqref{Equ9} reduces to the parabolic model
\begin{equation}
\label{Equ15} \partial _t u - \divy (\tilde{D}(y) \nabla _y u ) = 0,\;\;(t,y) \in \R_+ \times \R^m.
\end{equation}
The matrix field $\tilde{D}$ will appear as the orthogonal projection of the matrix field $D$ (with respect to some scalar product to be determined) on the subspace of matrix fields $A$ satisfying $[b\cdot \nabla _y, \divy(A \nabla _y)] = 0$. The field $\tilde{D}$ inherits the properties of $D$, like symmetry, positivity, etc.

Our paper is organized as follows. The main results are presented in Section \ref{ModMainRes}. Section \ref{AveOpe} is devoted to the interplay between the average operator and first and second order linear differential operators. In particular we justify the existence of the {\it averaged} matrix field $\tilde{D}$ associated to any field $D$ of symmetric, positive matrix. The first order approximation is justified in Section \ref{FirstOrdApp} and the second order approximation is discussed in Section \ref{SecOrdApp}. Several technical proofs are gathered in Appendix \ref{A}.

\section{Presentation of the models and main results}
\label{ModMainRes}
\noindent
We assume that the vector field $b :\R^m \to \R^m$ is smooth and divergence free
\begin{equation}
\label{Equ21} b \in W^{1,\infty}_{\mathrm{loc}} (\R^m),\;\;\divy b = 0
\end{equation}
with linear growth
\begin{equation}
\label{Equ22}
\exists \;C > 0\;\;\mbox{such that}\;\; |b(y)| \leq C (1 + |y|),\;\;y \in \R^m.
\end{equation}
We denote by $Y(s;y)$ the characteristic flow associated to $b$
\begin{equation}
\label{EquFlow} \frac{\md Y}{\md s} = b(Y(s;y)),\;\;Y(s;0) = y,\;\;s \in \R,\;\;y \in \R^m.
\end{equation}
Under the above hypotheses, this flow has the regularity $Y \in W^{1,\infty} _{\mathrm{loc}} (\R \times \R^m)$ and is measure preserving.

We concentrate on matrix fields $A(y) \in \loloc{}$ such that $[b(y) \cdot \nabla _y, \divy ( A(y) \nabla _y)] = 0$, let us say in $\dpri$. We check that the commutator between $b \cdot \nabla _y $ and $\divy (A \nabla _y)$ writes cf. Proposition \ref{ComSecOrd}
\[
[b(y) \cdot \nabla _y, \divy ( A(y) \nabla _y)] = \divy ( [b,A]\nabla _y)\;\;\mbox{in}\;\;\dpri
\]
where the bracket between $b$ and $A$ is given by
\[
[b,A] := (b \cdot \nabla _y) A - \partial _y b A (y) - A(y) \;^t \partial _y b,\;\;y \in \R^m.
\]
Several characterizations for the solutions of $[b,A] = 0$ in $\dpri$ are indicated in the Propositions \ref{MFI}, \ref{WMFI}, among which
\begin{equation}
\label{Equ16} A(Y(s;y)) = \partial _y Y (s;y) A(y) \;{^t \partial _y Y} (s;y),\;\;s\in \R,\;\;y \in \R ^m.
\end{equation}
We assume that there is a matrix field $P(y)$ such that 
\begin{equation}
\label{Equ56}  ^t P = P,\;\;P(y) \xi \cdot \xi >0,\;\;\xi \in \R^m,\;\; y \in \R^m,\;\;P^{-1}, P \in \ltloc{},\;\;[b,P]= 0 \;\mbox{in}\;\dpri.
\end{equation}
We introduce the set 
\[
H_Q = \{ A = A(y)\;:\; \inty{Q(y) A(y) : A(y) Q(y) } < +\infty\}
\]
where $Q = P ^{-1}$, and the scalar product 
\[
(A,B)_Q = \inty{QA:BQ},\;\;A, B \in H_Q.
\]
The equality \eqref{Equ16} suggests to introduce the family of applications $G(s): H_Q \to H_Q$, $s \in \R$, $G(s)A = (\partial _y Y )^{-1}(s; \cdot) A(Y(s;\cdot)) \;^t (\partial _y Y )^{-1}(s;\cdot)$ which is a $C^0$-group of unitary operators on $H_Q$ cf. Proposition \ref{Groupe}. This allows us to introduce $L$, the infinitesimal generator of $(G(s))_{s\in \R}$. The operator $L$ is skew-adjoint on $H_Q$ and its kernel coincides with $\{A \in H_Q\subset \loloc{} : [b,A] = 0\;\mbox{in} \; \dpri\}$ cf. Proposition \ref{PropOpeL}. The averaged matrix field denoted $\ave{D}_Q$, associated to any $D \in H_Q$ appears as the long time limit of the solution of
\begin{equation}
\label{Equ67}
\partial _t A - L(L(A)) = 0,\;\;t \in \R_+
\end{equation}
\begin{equation}
\label{Equ68} A(0) = D.
\end{equation}
The notation $\ave{\cdot}$ stands for the orthogonal projection (in $\lty{}$) on $\kerbg{}$.
\begin{thm}
\label{AveMatDif} Assume that \eqref{Equ21}, \eqref{Equ22}, \eqref{Equ56} hold true. Then for any $D \in H_Q \cap \liy{}$ the solution of \eqref{Equ67}, \eqref{Equ68} converges weakly in $H_Q$ as $t \to +\infty$ towards the orthogonal projection of $D$ on $\ker L$
\[
\lim _{t \to +\infty} A(t) = \ave{D}_Q\;\mbox{ weakly in }\;H_Q,\;\;\ave{D}_Q := \mathrm{Proj} _{\ker L } D.
\]
If $D$ is symmetric and positive, then so is the limit $\ave{D}_Q = \lim _{t \to +\infty} A(t)$, and satisfies
\begin{equation}
\label{Equ72} L (\ave{D}_Q) = 0,\;\;\nabla _y u \cdot \ave{D}_Q \nabla _y v = \ave{\nabla _y u \cdot D\nabla _y v},\;\;u, v \in H^1(\R^m) \cap \kerbg{}
\end{equation}
\begin{equation}
\label{Equ72Bis}\ave{\nabla _y u \cdot \ave{D}_Q \nabla _y (b \cdot \nabla _y \psi )} = 0,\;\;u \in  H^1(\R^m) \cap \kerbg{},\;\;\psi \in C^2_c (\R^m).
\end{equation}
\end{thm}
The first order approximation (for initial data not necessarily well prepared) is justified by
\begin{thm}
\label{MainResult1}  Assume that \eqref{Equ21}, \eqref{Equ22}, \eqref{Equ56}, \eqref{Equ26} hold true and that $D$ is a field of symmetric positive matrix, which belongs to $H_Q$. Consider a family of initial conditions $(\uein)_{\eps } \subset \lty$ such that $(\ave{\uein})_\eps$ converges weakly in $\lty{}$, as $\eps \searrow 0$, towards some function $\uin$. We denote by $\ue$ the solution of \eqref{Equ1}, \eqref{Equ2} and by $u$ the solution of
\begin{equation}
\label{Equ75}
\partial _t u - \divy ( \ave{D}_Q \nabla _y u ) = 0,\;\;t \in \R_+,\;\;y \in \R^m
\end{equation}
\begin{equation}
\label{Equ76}
u(0,y) = \uin (y),\;\;y \in \R^m
\end{equation}
where $\ave{D}_Q$ is associated to $D$, cf. Theorem \ref{AveMatDif}. Then we have the convergences
\[
\lime \ue = u\;\;\mbox{weakly} \star \mbox{ in } \litlty{}
\]
\[
\lime \nabla _y \ue = \nabla _y u\;\;\mbox{weakly} \mbox{ in } \lttlty{}.
\]
\end{thm}
The derivation of the second order approximation is more complicated and requires the computation of some other matrix fields. For simplicity, we content ourselves to formal results. The crucial point is to introduce the decomposition given by
\begin{thm}
\label{Decomposition} Assume that \eqref{Equ21}, \eqref{Equ22}, \eqref{Equ56}, \eqref{Equ26} hold true and that $L$ has closed range. Then, for any field of symmetric matrix $D \in H_Q$, there is a unique field of symmetric matrix $F \in \dom (L^2) \cap (\ker L )^\perp$ such that 
\[
- \divy ( D \nabla _y) = - \divy ( \ave{D}_Q \nabla _y ) + \divy (L^2 (F)\nabla _y )
\]
that is
\begin{align*}
& \inty{D \nabla _y u \cdot \nabla _y v }  - \inty{\ave{D}_Q\nabla _y u \cdot \nabla _y v } \\
& = \inty{L(F) \nabla _y u \cdot \nabla _y (b \cdot \nabla _y v)} + \inty{L(F) \nabla _y ( b \cdot \nabla _y u ) \cdot \nabla _y v} \\
& = - \inty{F \nabla _y ( b \cdot \nabla _y ( b \cdot \nabla _y u)) \cdot \nabla _y v} - 2 \inty{F \nabla _y ( b \cdot \nabla _y u) \cdot \nabla _y ( b \cdot \nabla _y v)}\\
& - \inty{F \nabla _y u \cdot \nabla _y ( b \cdot \nabla _y ( b \cdot \nabla _y v))}
\end{align*}
for any $u, v \in C^3_c(\R^m)$.
\end{thm}
After some computations we obtain, at least formally, the following model, replacing the hypothesis \eqref{Equ56} by the stronger one: there is a matrix field $R(y)$ such that 
\begin{equation}
\label{Equ90}
\det R(y)\neq 0,\;y \in \R^m,\;Q = {^t R} R \mbox{ and }P = Q^{-1} \in \ltloc{},\;b \cdot \nabla _y R + R \partial _y b = 0 \mbox{ in } \dpri. 
\end{equation}
\begin{thm}
\label{MainResult2} 
Assume that \eqref{Equ21}, \eqref{Equ22}, \eqref{Equ23}, \eqref{Equ26}, \eqref{Equ90} hold true and that $D$ is a field of symmetric positive matrix which belongs to $H_Q \cap \liy{}$. Consider a family of initial conditions $(\uein)_\eps \subset \lty{}$ such that $(\frac{\ave{\uein} - \uin }{\eps} ) _{\eps >0}$ converges weakly in $\lty{}$, as $\eps \searrow 0$, towards a function $\vin{}$, for some function $\uin \in \kerbg{}$. Then, a second order approximation for \eqref{Equ1} is provided by
\begin{equation}
\label{IntroEqu87}\partial _t \tue - \divy ( \ave{D}_Q \nabla _y \tue) + \eps [ \divy ( \ave{D}_Q \nabla _y ), \divy (F \nabla _y ) ]\tue - \eps S(\tue) = 0,\;\;(t,y) \in \R_+ \times \R^m
\end{equation}
\begin{equation}
\label{NewIC} \tue (0,y) = \uin (y) + \eps ( \vin (y) + \win (y)),\;\;\win = \divy ( F \nabla _y \uin),\;\;y \in \R ^m
\end{equation}
for some fourth order linear differential operator $S$, see Proposition \ref{DifOpe}, and the matrix field $F$ given by Theorem \ref{Decomposition}.
\end{thm}

\section{The average operator}
\label{AveOpe}
\noindent
We assume that the vector field $b : \R^m \to \R^m$ satisfies \eqref{Equ21}, \eqref{Equ22}. We consider the linear operator $u \to b \cdot \nabla _y u = \divy(ub)$ in $\lty{}$, whose domain is defined by
\[
\dom (b \cdot \nabla _y ) = \{ u \in \lty{} \;:\; \divy(ub) \in \lty\}.
\]
It is well known that 
\[
\kerbg = \{ u \in \lty{}\;:\; u (Y(s;\cdot)) = u (\cdot), \;s \in \R\}.
\]
The orthogonal projection on $\kerbg{}$ (with respect to the scalar product of $\lty{}$), denoted by $\ave{\cdot}$, reduces to average along the characteristic flow $Y$ cf. \cite{BosTraSin} Propositions 2.2, 2.3.
\begin{pro}
\label{AverageOperator} For any function $u \in \lty{}$ the family $\ave{u}_T : = \frac{1}{T} \int _0 ^T u (Y(s;\cdot))\md s, T>0$ converges strongly in $\lty{}$, when $T \to + \infty$, towards the orthogonal projection of $u$ on $\kerbg{}$
\[
\lim _{T \to +\infty} \ave{u}_T = \ave{u},\;\;\ave{u} \in \kerbg{} \;\mbox{and} \; \inty{(u - \ave{u}) \varphi } = 0,\;\forall\; \varphi \in \kerbg{}.
\]
\end{pro}
Since $b \cdot \nabla _y$ is antisymmetric, one gets easily
\begin{equation}
\label{Equ24} \overline{\ran (b \cdot \nabla _y ) } = (\kerbg{} ) ^\perp = \ker ( \mathrm{Proj}_{\kerbg{}} ) = \ker \ave{\cdot}.
\end{equation}
\begin{remark}
\label{DetFun} If $u \in \lty{}$ satisfies $\inty{u(y) b \cdot \nabla _y \psi } = 0, \forall \;\psi \in C^1 _c (\R^m)$ and $\inty{u\varphi }= 0, \forall \; \varphi \in \kerbg{}$, then $u = 0$. Indeed, as $u \in \lty{} \subset \loloc{}$, the first condition says that $b \cdot \nabla _y u = 0$ in $\dpri{}$ and thus $u \in \kerbg{}$. Using now the second condition with $\varphi = u$ one gets $\inty{u^2} = 0$ and thus $u = 0$.
\end{remark}
In the particular case when $\ran (b \cdot \nabla _y)$ is closed, which is equivalent to the Poincar\'e inequality (cf. \cite{Brezis} pp. 29)
\begin{equation}
\label{Equ23} \exists\;C_P >0\;:\; \left ( \inty{(u - \ave{u})^2}\right ) ^{1/2} \leq C_P \left ( \inty{(b \cdot \nabla _y u ) ^2} \right ) ^{1/2},\;\;u \in \dom (b \cdot \nabla _y)
\end{equation}
\eqref{Equ24} implies the solvability condition 
\[
\exists \; u \in \dom ( b \cdot \nabla _y ) \;\mbox{ such that }\; b \cdot \nabla _y u = v\;\mbox{ iff } \ave{v} = 0.
\]
If $\|\cdot \|$ stands for the $\lty{}$ norm we have
\begin{pro}
\label{Inverse}
Under the hypothesis \eqref{Equ23}, $b \cdot \nabla _y $ restricted to $\ker \ave{\cdot}$ is one to one map onto $\ker \ave{\cdot}$. Its inverse, denoted $(b \cdot \nabla _y )^{-1}$, belongs to ${\cal L}(\ker \ave{\cdot}, \ker \ave{\cdot})$ and 
\[
\|(b \cdot \nabla _y ) ^{-1} \|_{{\cal L}(\ker \ave{\cdot}, \ker \ave{\cdot})} \leq C_P.
\]
\end{pro}
Another operator which will play a crucial role is ${\cal T } = - \divy (b \otimes b \nabla _y)$ whose domain is
\[
\dom ({\cal T}) = \{ u \in \dom (b \cdot \nabla _y)\;:\; b \cdot \nabla _y u \in \dom ( b \cdot \nabla _y  )\}.
\]
The operator ${\cal T}$ is self-adjoint and under the previous hypotheses, has the same kernel and range as $b\cdot \nabla _y$.

\begin{pro}
\label{KerRanTau} Under the hypotheses \eqref{Equ21}, \eqref{Equ22}, \eqref{Equ23} the operator ${\cal T}$ satisfies
\[
\ker {\cal T} = \kerbg,\;\;\ran {\cal T} = \ran (b \cdot \nabla _y ) = \ker \ave{\cdot}
\]
and $\| u - \ave{u}\| \leq C_P ^2 \|{\cal T} u \|,u \in \dom ({\cal T})$.
\end{pro}
\begin{proof}
Obviously $\kerbg{} \subset \ker {\cal T}$. Conversely, for any $u \in \ker {\cal T}$ we have $\inty{\;(\bg u )^2} = \inty{\;u {\cal T}u} = 0$ and therefore $ u \in \kerbg{}$.

Clearly $\ran {\cal T} \subset \ran ( \bg{}) = \ker \ave{\cdot}$. Consider now $w \in \ker \ave{\cdot} = \ran ( \bg{})$. By Proposition \ref{Inverse} there is $v \in \ker \ave{\cdot} \cap \dom (\bg)$ such that $\bg v = w$. Applying one more time Proposition \ref{Inverse}, there is $ u \in \ker \ave{\cdot} \cap \dom (\bg)$ such that $\bg u = v$. We deduce that $u \in \dom {\cal T}, w = {\cal T}(-u)$. Finally, for any $u \in \dom {\cal T}$ we apply twice the Poincar\'e inequality, taking into account that $\ave{\bg u } = 0$
\[
\| u - \ave{u}\| \leq C_P \|\bg u \| \leq C_P ^2 \|{\cal T} u \|.
\]
\end{proof}
\begin{remark}
\label{AveLone}
The average along the flow of $b$ can be defined in any Lebesgue space $L^q (\R^m)$, $q \in [1,+\infty]$. We refer to \cite{BosTraSin} for a complete presentation of these results. 
\end{remark}

\subsection{Average and first order differential operators}
\label{FirstOrdDiffOpe}
\noindent
We are looking for first order derivations commuting with the average operator. Recall that the commutator $[\xi \cdot \nabla _y, \eta \cdot \nabla _y]$ between two first order differential operators is still a first order differential operator, whose vector field, denoted by $[\xi, \eta]$, is given by the Poisson bracket between $\xi$ and $\eta$
\[
[\xi \cdot \nabla _y, \eta \cdot \nabla _y]:= \xi \cdot \nabla _y ( \eta \cdot \nabla _y ) - \eta \cdot \nabla _y ( \xi \cdot \nabla _y ) = [\xi, \eta] \cdot \nabla _y 
\]
where $[\xi, \eta] = (\xi \cdot \nabla _y ) \eta - ( \eta \cdot \nabla _y ) \xi$. The two vector fields $\xi$ and $\eta$ are said in involution iff their Poisson bracket vanishes. 

Assume that $c(y)$ is a smooth vector field, satisfying $c(Y(s;y)) = \partial _y Y (s;y) c(y), s\in \R, y \in \R^m$, where $Y$ is the flow of $b$ (not necessarily divergence free here). Taking the derivative with respect to $s$ at $s = 0$ yields $(b \cdot \nabla _y ) c = \partial _y b \;c(y)$, saying that $[b,c] = 0$. Actually the converse implication holds true and we obtain the following characterization for vector fields in involution, which is valid in distributions as well (see Appendix \ref{A} for proof details). 
\begin{pro}
\label{VFI} Consider $b \in W^{1,\infty}_{\mathrm{loc}} (\R^m)$ (not necessarily divergence free), with linear growth and $c \in \loloc{}$. Then $(b \cdny) c - \partial _y b \;c = 0$ in $\dpri$ iff 
\begin{equation}
\label{Equ34} c (Y(s;y)) = \partial _y Y(s;y)  c(y),\;\;s\in \R,\;\;y \in \R^m.
\end{equation}
\end{pro}
We establish also weak formulations characterizing the involution between two fields, in distribution sense (see Appendix \ref{A} for the proof). The notation $w_s$ stands for $w \circ Y(s;\cdot)$.
\begin{pro}
\label{WVFI} Consider $b \in W^{1,\infty}_{\mathrm{loc}} (\R^m)$, with linear growth and zero divergence and $c \in \loloc{}$. Then the following statements are equivalent\\
1. 
\[[b,c] = 0 \;\mbox{in}\; \dpri{}
\]
2. 
\begin{equation}
\label{Equ41}
\inty{(c \cdny u )v_{-s} } = \inty{(c\cdny u_s) v },\;\;\forall \;u, v \in C^1_c(\R^m)
\end{equation}
3.
\begin{equation}
\label{Equ42} \inty{c \cdny u \;b \cdny v } + \inty{c \cdny (b \cdny u ) v } = 0,\;\;\forall\; u \in C^2 _c (\R^m),\;\;v \in C^1 _c (\R^m).
\end{equation}
\end{pro}
\begin{remark}
\label{VecDiv}
If $[b,c]=0$ in $\dpri{}$, applying \eqref{Equ41} with $v = 1$ on the support of $u_s$ (and therefore $v_{-s} = 1$ on the support of $u$) yields
\[
\inty{c \cdny u} = \inty{c\cdny u_s},\;\;u \in C^1_c(\R^m)
\]
saying that $\divy c$ is constant along the flow of $b$ (in $\dpri{}$). 
\end{remark}
We claim that for vector fields $c$ in involution with $b$, the derivation $c \cdny $ commutes with the average operator. 
\begin{pro}
\label{AveComFirstOrder} Consider a vector field $c \in \loloc{}$ with bounded divergence, in involution with $b$, that is $[b,c] = 0$ in $\dpri{}$. Then the operators $u \to c \cdny u$, $u \to \divy(uc)$ commute with the average operator {\it i.e.,} for any $u \in \dom ( c\cdny )= \dom (\divy (\cdot \;c))$ we have $\ave{u} \in \dom ( c\cdny )= \dom (\divy (\cdot \;c))$ and
\[
\ave{c \cdny u} = c\cdny \ave{u},\;\;\ave{\divy(uc) } = \divy ( \ave{u}c).
\]
\end{pro}
\begin{proof}
Consider $u \in \dom (c \cdny ), s \in \R$ and $\varphi \in C^1 _c (\R^m)$. We have
\begin{align}
\label{Equ43} \inty{u_s c \cdny \varphi } & = \inty{u (c \cdny \varphi)_{-s}} \\
& = \inty{u (c \cdny ) \varphi _{-s} } \nonumber \\
& = - \inty{\divy (uc) \varphi _{-s}} \nonumber \\
& = - \inty{(\divy (uc))_s \varphi (y)} \nonumber 
\end{align}
saying that $u _s \in \dom ( c \cdny ) = \dom ( \divy ( \cdot \;c ))$ and $ \divy (u_s c ) = ( \divy (uc))_s$. We deduce $c \cdny u_s = (c \cdny u )_s$ cf. Remark \ref{VecDiv}. Integrating \eqref{Equ43} with respect to $s$ between $0$ and $T>0$ one gets
\begin{align*}
\inty{\frac{1}{T} \int _0 ^T  u_s \md s \;c \cdny \varphi } & = \frac{1}{T} \int _0 ^T  \inty{u_s c \cdny \varphi }\md s \\
& = - \frac{1}{T} \int _0 ^T \inty{(\divy (uc ))_s \varphi (y) }\md s \\
& = - \inty{\frac{1}{T}\int _0 ^T  ( \divy (uc))_s \md s \;\varphi (y) }.
\end{align*}
By Proposition \ref{AverageOperator} we know that $\frac{1}{T} \int _0 ^T u_s \md s \to \ave{u}$ and $\frac{1}{T}\int _0 ^T  (\divy (uc))_s \md s \to \ave{\divy (uc)}$ strongly in $\lty{}$, when $T \to +\infty$, and thus we obtain 
\[
\inty{\ave{u} c \cdny \varphi } = - \inty{\ave{\divy(uc)} \varphi (y) }
\]
saying that $\ave{u} \in \dom ( c \cdny )$ and $\divy (\ave{u}c) = \ave{\divy (uc)}$, $c \cdny \ave{u} = \ave{c \cdny u }$.
\end{proof}

\subsection{Average and second order differential operators}
\label{SecondOrdDiffOpe}
\noindent
We investigate the second order differential operators $- \divy (A(y) \nabla _y)$ commuting with the average operator along the flow of $b$, where $A(y)$ is a smooth field of symmetric matrix. Such second order operators leave invariant $\kerbg{}$. Indeed, for any $u \in \dom (- \divy (A(y) \nabla _y)) \cap \kerbg{}$ we have
\[
- \divy (A(y) \nabla _y u ) = - \divy (A(y) \ave{u}) = \ave{ - \divy (A(y) \nabla _y u )} \in \kerbg{}.
\]
For this reason it is worth considering the operators $- \divy (A(y) \nabla _y )$ commuting with $b \cdny$. A straightforward computation shows that
\begin{pro}
\label{ComSecOrd} Consider a divergence free vector field $b \in W^{2,\infty} (\R^m)$ and a matrix field $A \in W^{2,\infty} (\R^m)$. The commutator between $b \cdny $ and $- \divy (A(y) \nabla _y)$ is still a second order differential operator
\[
[b\cdny, - \divy(A\nabla _y )]  = - \divy ([b,A] \nabla _y )
\]
whose matrix field, denoted by $[b,A]$, is given by
\[
[b,A] = (b \cdny )A - \dyb A(y) - A(y)\; {^t \dyb},\;\;y \in \R^m.
\]
\end{pro}
\begin{remark} 
We have the formula ${^t [b,A]} = [b, {^t A}]$. In particular if $A(y)$ is a field of symmetric (resp. anti-symmetric) matrix, the field $[b,A]$ has also symmetric (resp. anti-symmetric) matrix.
\end{remark}
As for vector fields in involution, we have the following characterization (see Appendix \ref{A} for proof details).
\begin{pro}
\label{MFI} Consider $b \in W^{1,\infty}_{\mathrm{loc}} (\R^m)$ (not necessarily divergence free) with linear growth and $A(y) \in \loloc{}$. Then $[b,A] = 0$ in $\dpri{}$ iff
\begin{equation}
\label{Equ35} A(\ysy) = \dyy A(y) \;{^t \dyy},\;\;s\in \R,\;\;y \in \R^m.
\end{equation}
\end{pro}
For fields of symmetric matrix we have the weak characterization (see Appendix \ref{A} for the proof).
\begin{pro}
\label{WMFI} Consider $b \in W^{1,\infty}_{\mathrm{loc}} (\R^m)$ with linear growth, zero divergence and $A \in \loloc{}$ a field of symmetric matrix. Then the following statements are equivalent\\
1. 
\[
[b,A] = 0 \;\mbox{ in } \; \dpri{}.
\]
2. 
\[
\inty{A(y) \nabla _y u_s \cdot \nabla _y v_s } = \inty{A(y) \nabla _y u \cdot \nabla _y v }
\]
for any $s \in \R$, $u, v \in C^1 _c ( \R^m)$.\\
3. 
\[
\inty{A(y) \nabla _y ( b \cdny u ) \cdot \nabla _y v } + \inty{A(y) \nabla _y u \cdot \nabla _y ( b \cdny v ) } = 0
\]
for any $u, v \in C^2 _c (\R^m)$.
\end{pro}

We consider the (formal) adjoint of the linear operator $A \to [b,A]$, with respect to the scalar product $(U,V) = \inty{U(y) : V(y)}$, given by
\[
Q \to - (b \cdny ) Q - {^t \dyb} Q(y) - Q(y) \dyb 
\] 
when $\divy b = 0$. The following characterization comes easily and the proof is left to the reader.
\begin{pro}
\label{AdyMatFieInv} Consider $b \in W^{1,\infty}_{\mathrm{loc}} (\R^m)$, with linear growth and $Q \in L^1 _{\mathrm{loc}} (\R^m)$. Then $- (b \cdny ) Q - {^t \dyb} Q(y) - Q(y) \dyb = 0$ in $\dpri {}$ iff
\begin{equation}
\label{Equ36} Q(\ysy) = {^t \partial _y Y }^{-1}(s;y) Q(y) \partial _y Y   ^{-1}(s;y),\;\;s\in \R,\;\;y \in \R^m.
\end{equation}
\end{pro}
\begin{remark}
\label{InverseQ}
If $Q(y)$ satisfies \eqref{Equ36} and is invertible for any $y \in \R^m$ with $Q^{-1} \in L^1 _{\mathrm{loc}}(\R^m)$, then $Q^{-1} (\ysy) = \dyy Q^{-1} (y) {^t \dyy}$, $s \in \R, y \in \R^m$ and therefore $[b,Q^{-1}] = 0$ in $\dpri{}$. If $P(y)$ satisfies \eqref{Equ35} and is invertible for any $y \in \R^m$, then 
\[
P^{-1} (\ysy) = {^t \partial _y Y} ^{-1} (s;y)  P ^{-1} (y)  \partial _y Y ^{-1} (s;y),\;\;s \in \R,\;\; y \in \R^m
\]
and therefore $- (b \cdny ) P - {^t \dyb} P(y) - P(y) \dyb = 0$ in $\dpri{}$. 
\end{remark}
As for vector fields in involution, the matrix fields in involution with $b$ generate second order differential operators commuting with the average operator.
\begin{pro}
\label{AveComSecondOrder} Consider a matrix field $A \in \loloc{}$ such that $\divy A \in \loloc{}$ and $[b,A] = 0$ in $\dpri{}$. Therefore the operator $u \to - \divy (A \nabla _y u )$ commutes with the average operator {\it i.e.,} for any $u \in \dom ( - \divy (A \nabla _y ))$ we have $\ave{u} \in \dom ( - \divy (A \nabla _y ))$ and 
\[
-\ave{\divy (A \nabla _y u )} = - \divy (A \nabla _y \ave{u}).
\]
\end{pro}
\begin{proof}
Consider $u \in \dom( - \divy (A\nabla _y )) = \{w \in \lty{}: -\divy (A\nabla _y w ) \in \lty{}\}$. For any $s \in \R, \varphi \in C^2 _c (\R^m)$ we have
\begin{equation}
\label{Equ51} - \inty{u_s \;\divy ( \;{^t A} \nabla _y \varphi ) } = - \inty{u \;( \divy ( \;{^t A} \nabla _\varphi ))_{-s}}.
\end{equation}
By the implication $1.\implies 2.$ of Proposition \ref{WMFI} (which does not require the symmetry of $A(y)$) we know that
\[
\inty{{^t A } \nabla _y \varphi \cdot \nabla _y \psi _s } = \inty{{^t A} \nabla _y \varphi _{-s} \cdot \nabla _y \psi }
\]
for any $\psi \in C^2 _c (\R^m )$. We deduce that 
\[
- \inty{\divy ( {^t A} \nabla _y \varphi ) \psi _s } = - \inty{\divy ( {^t A } \nabla _y \varphi _{-s} ) \psi }
\]
and thus $(\divy ( {^t A} \nabla _y \varphi ))_{-s} = \divy ( {^t A} \nabla _y \varphi _{-s})$. Combining with $\eqref{Equ51}$ yields
\begin{eqnarray}
\label{Equ52} - \inty{\;u_s \divy ( {^t A} \nabla _y \varphi )} & = - \inty{\;u\; \divy ( {^t A} \nabla _y \varphi _{-s})} \\
& = - \inty{\;\divy (A \nabla _y u ) \varphi _{-s}} \nonumber \\
& = - \inty{\;(\divy (A \nabla _y u ))_s \varphi (y)} \nonumber 
\end{eqnarray}
saying that $u_s \in \dom ( - \divy (A \nabla _y ))$ and
\[
- \divy ( A \nabla _y u_s)  = ( - \divy (A \nabla _y u ))_s.
\]
Integrating \eqref{Equ52} with respect to $s$ between $0$ and $T$ we obtain
\[
\inty{\frac{1}{T} \int _0 ^T u_s \;\md s\; \divy ( {^t A } \nabla _y \varphi )} = \inty{\frac{1}{T} \int _0 ^T (\divy (A \nabla _y u))_s \;\md s \;\varphi (y)}.
\]
Letting $T \to +\infty$ yields
\[
\inty{\ave{u} \divy ( {^t A }\nabla _y \varphi ) } = \inty{\ave{\divy (A \nabla _y u )} \varphi (y)}
\]
and therefore $\ave{u} \in \dom ( \divy (A \nabla _y ))$, $\divy (A \nabla _y \ave{u}) = \ave{\divy (A \nabla _y u )}$. 
\end{proof}

\subsection{The averaged diffusion matrix field}
\label{AveDifMatFie}
\noindent
We are looking for the limit, when $\eps \to 0$, of \eqref{Equ1}, \eqref{Equ2}. We expect that the limit $u = \lime \ue $ satisfies \eqref{Equ7}, \eqref{Equ8}. By \eqref{Equ7} we deduce that at any time $t \in \R_+$, $u(t,\cdot) \in \kerbg{}$. Observe also that $\divy(b \otimes b \nabla _y u^1) = b \cdny (b \cdny u^1) \in \ran ( b \cdny ) \subset \ker \ave{\cdot}$ and therefore the closure for $u$ comes by applying the average operator to \eqref{Equ8} and by noticing that $\ave{\partial _t u } = \partial _t \ave{u} = \partial _t u $
\begin{equation}
\label{Equ54} \partial _t u - \ave{\divy ( D \nabla _y u )} = 0,\;\;t\in \R_+,\;\;y \in \R^m.
\end{equation}
At least when $[b,D] = 0$, we know by Proposition \ref{AveComSecondOrder} that 
\[
\ave{\divy (D \nabla _y u)} = \divy (D \nabla _y \ave{u}) = \divy (D \nabla _y u)
\]
and \eqref{Equ54} reduces to the diffusion equation associated to the matrix field $D(y)$. Nevertheless, even if $[b,D] \neq 0$, \eqref{Equ54} behaves like a diffusion equation. More exactly the $\lty{}$ norm of the solution decreases with a rate proportional to the $\lty{}$ norm of its gradient under the hypothesis \eqref{Equ3}
\begin{align*}
\frac{1}{2}\frac{\md }{\md t} \inty{(u(t,y))^2} & = \inty{\ave{\divy ( D \nabla _y u )} u (t,y) } \\
& = \inty{\divy ( D \nabla _y u ) u }\\
& = - \inty{D \nabla _y u \cdot \nabla _y u }\\
& = - \inty{(D + b\otimes b ) : \nabla _y u \otimes \nabla _y u } \\
& \leq - d \inty{|\nabla _y u (t,y) |^2 }.
\end{align*}
We expect that, under appropriate hypotheses, \eqref{Equ54} coincides with a diffusion equation, corresponding to some {\it averaged} matrix field ${\cal D}$, that is
\begin{equation}
\label{Equ55}
\exists \; {\cal D} (y)\;:\; [b, {\cal D}] = 0\;\mbox{ and } \; \ave{- \divy ( D \nabla _y u )} = - \divy ( {\cal D} \nabla _y u ),\;\;\forall \;u \in \kerbg{}.
\end{equation}
It is easily seen that in this case the limit model \eqref{Equ54} reduces to
\[
\partial _t u - \divy ( {\cal D}\nabla _y u ) = 0,\;\;t \in \R_+,\;\;y \in \R^m.
\]
In this section we identify sufficient conditions which guarantee the existence of the matrix field ${\cal D}$. We will see that it appears as the long time limit of the solution of another parabolic type problem, whose initial data is $D$, and thus as the orthogonal projection of the field $D(y)$ (with respect to some scalar product to be defined) on a subset of $\{A \in \loloc{}:[b,A] = 0\mbox{ in } \dpri{}\}$. We assume that \eqref{Equ56} holds true. We introduce the set
\[
H_Q = \{A = A(y)\;:\; \inty{Q(y)A(y) : A(y) Q(y)} < +\infty\}
\]
where $Q = P^{-1}$ and the bilinear application 
\[
(\cdot, \cdot)_Q : H_Q \times H_Q \to \R,\;\;(A,B)_Q = \inty{Q(y)A(y):B(y)Q(y)}
\]
which is symmetric and positive definite. Indeed, for any $A \in H_Q$ we have
\[
(A,A)_Q = \inty{Q^{1/2}AQ^{1/2} : Q^{1/2}AQ^{1/2}} \geq 0
\]
with equality iff $Q^{1/2}AQ^{1/2}= 0$ and thus iff $A = 0$. The set $H_Q$ endowed with the scalar product $(\cdot, \cdot)_Q$ becomes a Hilbert space, whose norm is denoted by $|A|_Q = (A, A)_Q ^{1/2}, A \in H_Q$. Observe that $H_Q \subset \{A(y):A \in \loloc{}\}$. Indeed, if for any matrix $M$ the notation $|M|$ stands for the norm subordonated to the euclidian norm of $\R^m$
\[
|M| = \sup _{\xi \in \R^m \setminus \{0\}} \frac{|M\xi|}{|\xi|} \leq ( M : M ) ^{1/2}
\]
we have for a.a. $y \in \R^m$
\begin{eqnarray}
\label{Equ57} |A(y)| & = & \sup _{\xi, \eta \neq 0} \displaystyle \frac{A(y) \xi \cdot \eta}{|\xi|\;|\eta|} \\
& = & \sup _{\xi, \eta \neq 0} \displaystyle \frac{Q^{1/2}AQ^{1/2} P ^{1/2}\xi \cdot P^{1/2} \eta}{|P^{1/2} \xi|\;|P^{1/2} \eta|}\;\frac{|P^{1/2}\xi|}{|\xi|}\;\frac{|P^{1/2}\eta|}{|\eta|} \nonumber \\
& \leq & |Q^{1/2}AQ^{1/2} |\;|P^{1/2}|^2 \nonumber \\
& \leq & ( Q^{1/2}AQ^{1/2}:Q^{1/2}AQ^{1/2}) ^{1/2} \;|P|.\nonumber 
\end{eqnarray}
We deduce that for any $R>0$
\[
\int_{B_R} |A(y)|\;\md y \leq \int _{B_R} ( Q^{1/2}AQ^{1/2}:Q^{1/2}AQ^{1/2}) ^{1/2} \;|P|\;\md y \leq (A,A)_Q ^{1/2} \left ( \int _{B_R} |P(y)|^2 \;\md y \right ) ^{1/2}.
\]
\begin{remark}
\label{Ortho}
We know by Remark \ref{InverseQ} that $Q_s = {^t \partial _y  Y ^{-1} }(s;y) Q(y) \partial _y Y ^{-1}(s;y) $ which writes ${^t {\cal O}}(s;y) {\cal O}(s;y) = I$ where ${\cal O}(s;y) = Q_s ^{1/2} \dyy Q^{-1/2}$. Therefore the matrix ${\cal O}(s;y)$ are orthogonal and we have
\begin{equation}
\label{Equ58}
Q_s ^{1/2} \dyy Q^{-1/2} = {\cal O}(s;y) = {^t {\cal O}}^{-1} (s;y) = Q_s ^{-1/2} \;{^t \partial _y Y }^{-1} Q^{1/2}
\end{equation}
\begin{equation}
\label{Equ59} 
Q ^{-1/2} \;{^t \dyy} Q_s ^{1/2} = {^t {\cal O}}(s;y) = { {\cal O}}^{-1} (s;y) = Q ^{1/2} { \partial _y Y }^{-1} Q_s^{-1/2}.
\end{equation}
\end{remark}
\begin{pro}
\label{Groupe} The family of applications $A \to G(s)A : = \partial _y Y ^{-1} (s; \cdot) A_s \; {^t \partial _y Y } ^{-1} (s; \cdot)$ is a $C^0$-    group of unitary operators on $H_Q$.
\end{pro}
\begin{proof}
For any $A\in H_Q$ observe, thanks to \eqref{Equ59}, that
\begin{align*}
\left | \partial _y Y ^{-1}(s; \cdot) A_s {^t\partial _y Y ^{-1}(s; \cdot) }\right | ^2 _Q & 
= \!\!\inty{Q^{1/2}\partial _y Y ^{-1} A_s {^t \partial _y Y ^{-1}}Q^{1/2}:Q^{1/2}\partial _y Y ^{-1} A_s {^t \partial _y Y ^{-1}}Q^{1/2}}\\
& = \!\!\inty{\!\!\!\!{^t {\cal O}} (s;y) Q_s ^{1/2} A_s Q_s ^{1/2} {\cal O}(s;y) \!:\! {^t {\cal O}} (s;y) Q_s ^{1/2} A_s Q_s ^{1/2} {\cal O}(s;y)}\\
& = \inty{Q_s ^{1/2} A_s Q_s ^{1/2} : Q_s ^ {1/2} A_s Q_s ^{1/2}}\\
& = \inty{Q^{1/2}AQ^{1/2} : Q^{1/2}AQ^{1/2}} \\
& = |A|^2 _Q.
\end{align*}
Clearly $G(0)A = A, A\in H_Q$ and for any $s, t \in \R$ we have
\begin{align*}
G(s) G(t) A & = \partial _y Y ^{-1} (s;\cdot) (G(t)A)_s {^t \partial _y Y ^{-1} (s;\cdot)}\\
& = \partial _y Y ^{-1} (s;\cdot) (\partial _y Y )^{-1} (t; Y(s;\cdot))(A_t)_s {^t (\partial _y Y )^{-1} (t; Y(s;\cdot))}{^t \partial _y Y ^{-1} (s;\cdot)} \\
& = \partial _y Y ^{-1} (t + s;\cdot)A_{t+s} {^t \partial _y Y ^{-1} (t + s;\cdot)} = G(t+s) A,\;\;A \in H_Q.
\end{align*}
It remains to check the continuity of the group, {i.e.,} $\lim _{s \to 0 } G(s)A = A$ strongly in $H_Q$ for any $A \in H_Q$. For any $s \in \R$ we have
\begin{align*}
|G(s) A - A|^2 _Q = |G(s)A|^2 _Q + |A|^2 _Q - 2 ( G(s)A, A)_Q = 2|A|^2 _Q - 2 (G(s)A, A)_Q
\end{align*}
and thus it is enough to prove that $\lim _{s \to 0 } G(s)A = A$ weakly in $H_Q$. As $|G(s)| = 1$ for any $s \in \R$, we are done if we prove that $\lim _{s \to 0} (G(s)A, U)_Q = (A, U)_Q$ for any $U \in C^0 _c (\R^m) \subset H_Q$. But it is easily seen that $\lim _{s\to 0} G(-s)U = U$ strongly in $H_Q$, for $U \in C^0 _c (\R^m) $ and thus
\[
\lim _{s \to 0} ( G(s)A, U)_Q = \lim _{s \to 0} (A, G(-s)U)_Q = (A,U)_Q,\;\;U \in C^0 _c (\R^m).
\]
\end{proof}
We denote by $L$ the infinitesimal generator of the group $G$
\[
L:\dom(L) \subset H_Q \to H_Q,\;\;\dom L = \{ A\in H_Q\;:\; \exists \;\lim _{s \to 0} \frac{G(s)A-A}{s}\;\mbox{ in } \;H_Q\}
\]
and $L(A) = \lim _{s \to 0} \frac{G(s)A-A}{s}$ for any $A \in \dom(L)$.
Notice that $C^1 _c (\R^m) \subset \dom(L)$ and $L(A) = b \cdny A - \dyb A - A \;{^t \dyb}$, $A \in C^1 _c (\R^m)$ (use the hypothesis $Q \in \ltloc{}$ and the dominated convergence theorem). Observe also that the group $G$ commutes with transposition {\it i.e.} $G(s) \;{^t A} = {^t G(s)}A$, $s \in \R, A \in H_Q$ and for any $A \in \dom (L)$ we have
$^t A \in \dom (L)$, $L({^t A}) = {^t L(A)}$. 
The main properties of the operator $L$ are summarized below (when $b$ is divergence free).
\begin{pro}
\label{PropOpeL} $\;$\\
1. The domain of $L$ is dense in $H_Q$ and $L$ is closed.\\
2. The matrix field $A \in H_Q$ belongs to $\dom (L)$ iff there is a constant $C >0$ such that 
\begin{equation}
\label{Equ61} |G(s)A - A |_Q \leq C |s|,\;\;s \in \R.
\end{equation}
3. The operator $L$ is skew-adjoint.\\
4. For any $A \in \dom (L)$ we have
\[
- \divy (L(A) \nabla _y ) = b\cdny ( - \divy (A \nabla _y)) + \divy (A \nabla _y ( b \cdny ))\;\mbox{ in } \; \dpri{}
\]
that is
\[
\inty{L(A) \nabla _y u \cdot \nabla _y v } = - \inty{A \nabla _y u \cdot \nabla _y ( b \cdny v)} - \inty{A\nabla _y ( b \cdny u ) \cdot \nabla _y v }
\]
for any $u, v \in C^2 _c (\R^m)$. 
\end{pro}
\begin{proof}
1. The operator $L$ is the infinitesimal generator of a $C^0$-group, and therefore $\dom(L)$ is dense and $L$ is closed. \\
2.
Assume that $A \in \dom(L)$. We know that $\frac{\md }{\md s} G(s)A = L(G(s)A) = G(s)L(A)$ and thus 
\[
|G(s)A - A|_Q = \left | \int _0 ^t G(\tau) L(A)\;\md \tau\right |_Q \leq \left | \int _0 ^s |G(\tau)L(A)|_Q \;\md \tau \right | = |s| \;|L(A)|_Q,\;\;s \in \R.
\]
Conversely, assume that \eqref{Equ61} holds true. Therefore we can extract a sequence $(s_k)_k$ converging to $0$ such that
\[
\limk \frac{G(s_k) A - A}{s_k} = V \;\mbox{ weakly in } \;H_Q.
\]
For any $U \in \dom (L)$ we obtain
\[
\left ( \frac{G(s_k) A - A}{s_k}, U \right ) _Q = \left ( A, \frac{G(-s_k)U - U}{s_k} \right ) _Q
\]
and thus, letting $k \to +\infty$ yields
\begin{equation}
\label{Equ62} (V, U)_Q = - (A, L(U))_Q.
\end{equation}
But since $U \in \dom (L)$, all the trajectory $\{G(\tau)U:\tau \in \R\}$ is contained in $\dom(L)$ and $G(-s_k)U = U + \int _0 ^{-s_k}L(G(\tau)U)\md \tau$. We deduce 
\begin{align*}
(G(s_k)A - A, U)_Q & = \left ( A, \int _0 ^{-s_k} L(G(\tau)U ) \;\md \tau \right ) \\
& = \int _0 ^{-s_k}( A, L(G(\tau) U))_Q \;\md \tau \\
& = - \int _0 ^{-s_k}( V, G(\tau) U )_Q \;\md \tau \\
& = - \left ( V, \int _0 ^{-s_k} G(\tau)U\;\md \tau \right ) _Q.
\end{align*}
Taking into account that $\left | \int _0 ^{-s_k} G(\tau ) U \md \tau \right |_Q \leq |s_k| \;|U|_Q$ we obtain
\[
\left | \left ( \frac{G(s_k)A - A}{s_k}, U \right ) _Q \right | \leq |V|_Q|U|_Q,\;\;U \in \dom (L)
\]
and thus, by the density of $\dom (L)$ in $H_Q$ one gets
\[
\left | \frac{G(s_k)A - A}{s_k}  \right |_Q \leq |V|_Q,\;\;k \in \N.
\]
Since $V$ is the weak limit in $H_Q$ of $\left ( \frac{G(s_k)A - A}{s_k}  \right )_k$, we deduce that $\limk \frac{G(s_k)A - A}{s_k}  = V$ strongly in $H_Q$. As the limit $V$ is uniquely determined by \eqref{Equ62}, all the family $\left ( \frac{G(s)A - A}{s}  \right )_s$ converges strongly , when $s \to 0$, towards $V$ in $H_Q$ and thus $A \in \dom (L)$.\\
3. For any $U, V \in \dom (L)$ we can write
\[
(G(s)U - U, V)_Q + (U, V - G(-s)V)_Q = 0,\;\;s\in \R.
\]
Taking into account that 
\[
\lim _{s \to 0} \frac{G(s)U - U}{s} = L(U),\;\;\lim _{s \to 0} \frac{V - G(-s)V}{s} = L(V)
\]
we obtain $(L(U), V)_Q + (U, L(V))_Q = 0$ saying that $V\in \dom (L^\star)$ and $L^\star (V) = - L(V)$. Therefore $L \subset (-L^\star)$. It remains to establish the converse inclusion. Let $V \in \dom (L^\star)$, {\it i.e.,} $\exists C >0$ such that 
\[
|(L(U), V)_Q|\leq C|U|_Q,\;\;U \in \dom (L).
\]
For any $s \in \R$, $U \in \dom (L)$ we have
\[
(G(s)V - V , U)_Q = (V, G(-s)U - U)_Q = (V, \int _0 ^{-s}LG(\tau)U \;\md \tau )_Q = \int _0 ^{-s} (V, LG(\tau)U)_Q \;\md \tau
\]
implying 
\[
|(G(s) V - V , U )_Q|\leq C |s| \;|U|_Q,\;\;s\in \R.
\]
Therefore $|G(s)V - V|_Q \leq C |s|, s \in \R$ and by the previous statement $V \in \dom (L)$. Finally $\dom (L) = \dom (L^\star)$ and $L^\star (V) = - L(V), V \in \dom (L) = \dom (L^\star)$.\\
4. As $L$ is skew-adjoint, we obtain
\[
- \inty{L(A)\nabla _y u \cdot \nabla _y v } = - ( L(A), Q^{-1} \nabla _y v \otimes \nabla _y u Q^{-1}\;)_Q = ( A, L ( Q^{-1} \nabla _y v \otimes \nabla _y u Q^{-1})\;)_Q.
\]
Recall that $P = Q^{-1}$ satisfies $L(P) = 0$, that is, $G(s)P = P, s \in \R$ and thus
\begin{align*}
L(Q^{-1} \nabla _y v \otimes & \nabla _y u Q^{-1})  = \lim _{ s \to 0} \frac{G(s)P\nabla _y v \otimes \nabla _y u P - P \nabla _y v \otimes \nabla _y u P}{s} \\
& = \lim _{ s \to 0} \frac{\partial _y Y ^{-1} (s;\cdot) P_s (\nabla _y v )_s \otimes (\nabla _y u )_s P_s {^t \partial _y Y ^{-1}}(s;\cdot) - P \nabla _y v \otimes \nabla _y u P}{s}\\
& = \lim _{ s \to 0} \frac{P{^t \partial _y Y} (s;\cdot) (\nabla _y v )_s \otimes (\nabla _y u )_s \partial _y Y (s;\cdot)P - P \nabla _y v \otimes \nabla _y u P}{s} \\
& = \lim _{ s \to 0} \frac{P \nabla _y v_s \otimes \nabla _y u_s P -P \nabla _y v \otimes \nabla _y u P }{s} \\
& = P \nabla _y ( b \cdny v ) \otimes \nabla _y u P + P \nabla _y v \otimes \nabla _y ( b \cdny u ) P.
\end{align*}
Finally one gets
\begin{align*}
- \inty{L(A) \nabla _y u \cdot \nabla _y v } & = ( A, P \nabla _y ( b \cdny v ) \otimes \nabla _y u P) + P \nabla _y v \otimes \nabla _y ( b \cdny u )P)_Q \\
& = \inty{A\nabla _y u \cdot \nabla _y ( b \cdny v)} + \inty{A\nabla _y ( b \cdny u ) \cdot \nabla _y v }.
\end{align*}
\end{proof}
We claim that $\dom (L)$ is left invariant by some special (weighted with respect to the matrix field $Q$) positive/negative part functions. The notations $A^\pm$ stand for the usual positive/negative parts of a symmetric matrix $A$
\[
A^\pm = S \Lambda ^\pm \;{^t S},\;\;A = S\Lambda \;{^t S}
\]
where $\Lambda, \Lambda ^\pm $ are the diagonal matrix containing the eigenvalues of $A$ and the positive/negative parts of these eigenvalues respectively, and $S$ is the orthogonal matrix whose columns contain a orthonormal basis of eigenvectors for $A$. Notice that 
\[
A^+ : A^- = 0,\;\;A^+ - A^- = A,\;\;A^+ : A^+ + A^- : A^- = A: A.
\]
We introduce also the positive/negative part functions which associate to any field of symmetric matrix $A(y)$ the fields of symmetric matrix $A^{Q\pm}(y)$ given by
\[
Q^{1/2} A^{Q\pm} \;Q^{1/2} = (Q^{1/2} AQ^{1/2})^\pm.
\] 
Observe that $A^{Q+} - A^{Q-} = A$.
\begin{pro}
\label{InvPosNeg}$\;$\\
1. The applications $A \to A^{Q\pm}$ leave invariant the subset $\{A\in \dom (L): {^t A} = A\}$.\\
2. For any $A \in \dom (L), {^t A } = A$ we have
\[
(A^{Q+}, A^{Q-})_Q = 0,\;\;( L(A^{Q+}), L(A^{Q-}))_Q \leq 0.
\]
\end{pro}
\begin{proof}
1. Consider $A \in \dom (L), {^t A} = A$. It is easily seean that ${^t A^{Q\pm}} = A^{Q\pm}$ and
\begin{align*}
|A^{Q+}|^2 _Q  + |A^{Q-}|^2 _Q   & = \inty{(Q^{1/2}A Q^{1/2} ) ^ + : ( Q^{1/2} A Q^{1/2})^+}  \\
& + \inty{(Q^{1/2}A Q^{1/2} ) ^ - : ( Q^{1/2} A Q^{1/2})^-}\\
& = \inty{Q^{1/2}A Q^{1/2} : Q^{1/2} A Q^{1/2}} = |A|^2 _Q < +\infty
\end{align*}
and therefore $A ^{Q\pm} \in H_Q$. The positive/negative parts $A^{Q\pm}$ are orthogonal in $H_Q$
\[
( A^{Q+} , A^{Q-})_Q = \inty{(Q^{1/2}AQ^{1/2})^+ : (Q^{1/2}AQ^{1/2})^-} = 0.
\]
We claim that $A^{Q\pm}$ satisfies \eqref{Equ61}. Indeed, thanks to \eqref{Equ59} we can write, using the notation $X^{:2} = X : X$
\begin{align}
\label{Equ63} |G(s)A^{Q\pm}- A^{Q\pm}|^2 _Q & = 
\inty{\{ Q^{1/2} ( \partial _y Y ^{-1} (A ^{Q\pm})_s {^t \partial _y Y ^{-1}} - A^{Q\pm})Q^{1/2}    \} ^{:2}}\\
& = \inty{\{{^t {\cal O}} (s;y) Q_s ^{1/2} (A ^{Q\pm})_sQ_s ^{1/2} {\cal O}(s;y) - Q^{1/2} A^{Q\pm}Q^{1/2}   \}^{:2}} \nonumber \\
& = \inty{\{ {^t {\cal O}} (s;y) ( Q_s ^{1/2} A_s Q_s ^{1/2})^{\pm} {\cal O}(s;y) - (Q^{1/2}A Q^{1/2} ) ^{\pm} \} ^{:2}}. \nonumber 
\end{align}
Similarly we obtain 
\begin{equation}
\label{Equ64} |G(s)A - A|^2 _Q = \inty{\{{^t {\cal O}}(s;y) Q^{1/2}_s A_s Q^{1/2}_s {\cal O}(s;y) - Q^{1/2}AQ^{1/2}   \} ^{:2}}.
\end{equation}
We are done if we prove that for any symmetric matrix $U, V$ and any orthogonal matrix $R$ we have the inequality
\begin{equation}
\label{Equ65}
( \;{^t R } U ^{\pm} R - V ^\pm \;) : ( \;{^t R } U ^{\pm} R - V ^\pm \;)\leq ( \;{^t R } U  R - V \;):( \;{^t R } U  R - V \;).
\end{equation}
For the sake of the presentation, we consider the case of positive parts $U^+, V^+$. The other one comes in a similar way. The above inequality reduces to
\[
2 \;{^t R } U R : V - 2 \;{^t R } U ^+ R : V ^+ \leq {^t R } U ^- R : {^t R } U ^- R + V^- : V^-
\]
or equivalently, replacing $U$ by $ U^+ - U^-$ and $V$ by $V^+ - V^-$, to
\[
- 2 \;{^t R } U ^+ R : V^- - 2 \;{^t R } U ^- R : V^+ + 2 \;{^t R } U ^- R : V ^- \leq {^t R } U ^- R : {^t R } U ^- R + V ^- : V^-.
\]
It is easily seen that the previous inequality holds true, since ${^t R } U ^+ R : V^- \geq 0$, ${^t R } U ^- R : V^+ \geq 0$ and
\[
2 \;{^t R } U ^- R : V^- \leq 2 ( {^t R} U ^- R :  {^t R} U ^- R) ^{1/2} ( V^- : V^- ) ^{1/2} \leq {^t R} U ^- R :  {^t R} U ^- R + V^- : V^-.
\]
Combining \eqref{Equ63}, \eqref{Equ64} and \eqref{Equ65} with 
\[
U = Q^{1/2} _s A_s Q^{1/2}_s, \;\;V = Q^{1/2}AQ^{1/2},\;\;R = {\cal O}
\]
yields
\[
\sup _{s \neq 0} \frac{|G(s)A^{Q\pm} - A^{Q\pm} |_Q}{|s|} \leq \sup _{s \neq 0} \frac{|G(s)A - A|_Q}{|s|} \leq |L(A)|_Q 
\]
saying that $A^{Q\pm} \in \dom (L)$. \\
2. For any $A \in \dom (L)$, $^t A = A$ we can write
\begin{align*}
(A^{Q+}, A^{Q-})_Q & = \inty{Q^{1/2}A^{Q+}Q^{1/2}: Q^{1/2}A^{Q-}Q^{1/2}} \\
& = \inty{(Q^{1/2}AQ^{1/2})^+ : (Q^{1/2}AQ^{1/2})^-} = 0.
\end{align*}
Since $A^{Q\pm} \in \dom (L)$ we have
\[
L(A^{Q\pm}) = \lim _{s \to 0} \frac{G(s/2) A^{Q\pm} - G(-s/2) A^{Q\pm}}{s}
\]
and therefore, thanks to \eqref{Equ59}, we obtain
\begin{align*}
& ( L(A^{Q+}), L(A^{Q-}))_Q  = \lim _{s \to 0} \left (\frac{G(\frac{s}{2}) A^{Q+} - G(-\frac{s}{2}) A^{Q+}}{s}, \frac{G(\frac{s}{2}) A^{Q-} - G(-\frac{s}{2}) A^{Q-}}{s}   \right ) _Q\\
& = \lim _{s \to 0} \inty{\frac{Q^{1/2} (\;G(\frac{s}{2}) A^{Q+} - G(-\frac{s}{2}) A^{Q+} \;) Q^{1/2} }{s} : \frac{Q^{1/2} (\;G(\frac{s}{2}) A^{Q-} - G(-\frac{s}{2}) A^{Q-} \;) Q^{1/2}}{s} }\\
& = \lim _{s \to 0} \inty{ \frac{{^t {\cal O}(\frac{s}{2};y)}( Q^{1/2}_{\frac{s}{2}} A_{\frac{s}{2}} Q^{1/2}_{\frac{s}{2}})^+ {\cal O}(\frac{s}{2};y) - {^t {\cal O}(-\frac{s}{2};y)}( Q^{1/2}_{-\frac{s}{2}} A_{-\frac{s}{2}} Q^{1/2}_{-\frac{s}{2}})^+ {\cal O}(-\frac{s}{2};y)}{s} \\
& : \frac{{^t {\cal O}(\frac{s}{2};y)}( Q^{1/2}_{\frac{s}{2}} A_{\frac{s}{2}} Q^{1/2}_{\frac{s}{2}})^- {\cal O}(\frac{s}{2};y) - {^t {\cal O}(-\frac{s}{2};y)}( Q^{1/2}_{-\frac{s}{2}} A_{-\frac{s}{2}} Q^{1/2}_{-\frac{s}{2}})^- {\cal O}(-\frac{s}{2};y)}{s}}\\
& = - \lim _{s \to 0} \inty{\frac{{^t {\cal O}(\frac{s}{2};y)}( Q^{1/2}_{\frac{s}{2}} A_{\frac{s}{2}} Q^{1/2}_{\frac{s}{2}})^+ {\cal O}(\frac{s}{2};y) : {^t {\cal O}(-\frac{s}{2};y)}( Q^{1/2}_{-\frac{s}{2}} A_{-\frac{s}{2}} Q^{1/2}_{-\frac{s}{2}})^- {\cal O}(-\frac{s}{2};y)}{s^2}} \\
& - \lim _{ s \to 0}\inty{\frac{{^t {\cal O}(-\frac{s}{2};y)}( Q^{1/2}_{-\frac{s}{2}} A_{-\frac{s}{2}} Q^{1/2}_{-\frac{s}{2}})^+ {\cal O}(-\frac{s}{2};y) : {^t {\cal O}(\frac{s}{2};y)}( Q^{1/2}_{\frac{s}{2}} A_{\frac{s}{2}} Q^{1/2}_{\frac{s}{2}})^- {\cal O}(\frac{s}{2};y)}{s^2}} \\
& \leq 0
\end{align*}
since 
\[
{^t {\cal O}}(\pm s/2;\cdot) ( Q^{1/2} A Q^{1/2}) _{\pm s/2} ^\pm {\cal O}(\pm s/2;\cdot) \geq 0,\;\;{^t {\cal O}}(\mp s/2;\cdot) ( Q^{1/2} A Q^{1/2}) _{\mp s/2} ^\pm {\cal O}(\mp s/2;\cdot) \geq 0.
\]
\end{proof}
We intend to solve the problem \eqref{Equ67}, \eqref{Equ68} by using variational methods. We introduce the space $V_Q = \dom (L) \subset H_Q$ endowed with the scalar product
\[
((A, B))_Q = (A, B)_Q + (L(A), L(B))_Q,\;\;A, B \in V_Q.
\]
Clearly $(V_Q, ((\cdot, \cdot))_Q)$ is a Hilbert space (use the fact that $L$ is closed) and the inclusion $V_Q \subset H_Q$ is continuous, with dense image. The notation $\|\cdot \|_Q$ stands for the norm associated to the scalar product $((\cdot, \cdot))_Q$
\[
\|A\|^2 _Q = ((A, A))_Q = (A, A)_Q + (L(A), L(A))_Q = |A|^2 _Q + |L(A)|^2_Q,\;\;A\in V_Q.
\]
We introduce the bilinear form $\sigma : V_Q \times V_Q \to \R$
\[
\sigma (A, B) = (L(A), L(B))_Q,\;\;A, B \in V_Q.
\]
Notice that $\sigma$ is coercive on $V_Q$ with respect to $H_Q$
\[
\sigma (A, A) + |A|^2_Q = \|A\|^2_Q,\;\;A \in V_Q.
\]
By Theorems 1,2 pp. 620 \cite{DauLions88} we deduce that for any $D \in H_Q$ there is a unique variational solution for \eqref{Equ67}, \eqref{Equ68} that is $A \in C_b (\R_+; H_Q) \cap L^2 (\R_+; V_Q)$, $\partial _t A \in L^2 (\R_+; V_Q ^\prime)$
\[
A(0) = D,\;\;\frac{\md }{\md t } (A(t), U)_Q + \sigma (A(t), U) = 0,\;\;\mbox{in}\;\;\dpri{},\;\;\forall \;U \in V_Q.
\]
The long time limit of the solution of \eqref{Equ67}, \eqref{Equ68} provides the averaged matrix field in \eqref{Equ55}.
\begin{proof}
(of Theorem \ref{AveMatDif}) The identity
\[
\frac{1}{2}\frac{\md }{\md t} |A(t) |^2 _Q + |L(A(t))|^2 _Q = 0,\;\;t \in \R_+
\]
gives the estimates
\[
|A(t)|_Q \leq |D|_Q,\;\;t \in \R_+,\;\;\int _0 ^{+\infty} |L(A(t))|^2 _Q \;\md t\leq \frac{1}{2}|D|^2 _Q.
\]
Consider $(t_k)_k$ such that $t_k \to +\infty$ as $k \to +\infty$ and $(A(t_k))_k$ converges weakly towards some matrix field $X$ in $H_Q$. For any $U \in \ker L$ we have
\[
\frac{\md }{\md t} (A(t), U)_Q = 0,\;\;t \in \R_+
\]
and therefore
\begin{equation}
\label{Equ70} (\mathrm{Proj}_{\ker L} D, U)_Q = (D,U)_Q = (A(0), U)_Q = (A(t_k), U)_Q = ( X, U)_Q,\;\;U \in \ker L.
\end{equation}
Since $L(A) \in L^2 (R_+;H_Q)$ we deduce that $\limk L(A(t_k)) = 0$ strongly in $H_Q$. For any $V \in V_Q$ we have
\[
(X, L(V))_Q = \limk (A(t_k), L(V))_Q = - \limk (L(A(t_k)), V)_Q = 0.
\]
We deduce that $X \in \dom (L^\star) = \dom (L)$ and $L(X) = 0$, which combined with \eqref{Equ70} says that $X = \mathrm{Proj}_{\ker L} D$, or $X = \ave{D}_Q$. By the uniqueness of the limit we obtain $\lim _{t \to +\infty} A(t) = \mathrm{Proj}_{\ker L} D$ weakly in $H_Q$. Assume now that ${^t D } = D$. As $L$ commutes with transposition, we have $\partial _t {^t A} - L (L({^t A})) = 0$, ${^t A }(0) = D$. By the uniqueness we obtain ${^t A } = A$ and thus
\[
^t \ave{D}_Q = \;^t ( \mbox{w}-\lim _{t \to +\infty} A(t) ) = \mbox{w}-\lim _{t \to +\infty} {^t A(t)} = \mbox{w}-\lim _{t \to +\infty} A(t)= \ave{D}_Q.
\]
Suppose that $D\geq 0$ and let us check that $\ave{D}_Q \geq 0$. By Proposition \ref{InvPosNeg} we know that $A^{Q\pm}(t) \in V_Q$, $t \in \R_+$ and
\[
(A^{Q+}(t), A^{Q-}(t))_Q = 0,\;\;(L(A^{Q+}(t)), L(A^{Q-}(t)))_Q \leq  0,\;\;t \in \R_+.
\]
It is sufficient to consider the case of smooth solutions. Multiplying \eqref{Equ67} by $-A^{Q-}(t)$ one gets
\begin{align}
\label{Equ71} \frac{1}{2}\frac{\md }{\md t} |A^{Q-}(t) |^2 _Q + |L(A^{Q-}(t)|^2 _Q & = ( \partial _t A^{Q+}, A^{Q-}(t))_Q + (L(A^{Q+}(t)),L(A^{Q-}(t)) )_Q \\
& \leq  ( \partial _t A^{Q+}, A^{Q-}(t))_Q. \nonumber 
\end{align}
But for any $0 < h < t $ we have
\begin{align*}
(A^{Q+}(t) - A^{Q+}(t-h), A^{Q-}(t))_Q = - (A^{Q+}(t-h), A^{Q-}(t))_Q \leq 0
\end{align*}
and therefore $(\partial _t A^{Q+}(t), A^{Q-}(t))_Q \leq 0 $. Observe that $Q^{1/2}A^{Q-}(0)Q^{1/2} = (Q^{1/2} D Q^{1/2})^- = 0$, since $Q^{1/2} D Q^{1/2}$ is symmetric and positive. Thus $A^{Q-}(0) = 0$, and from \eqref{Equ71} we obtain
\[
\frac{1}{2} |A^{Q-}(t) |^2 _Q \leq \frac{1}{2}|A^{Q-}(0)|^2 _Q = 0
\]
implying that $Q^{1/2} A(t) Q^{1/2} \geq 0$ and $A(t) \geq 0$, $t \in \R_+$. 
Take now any $U \in H_Q$, ${^t U } = U$, $U \geq 0$. By weak convergence we have
\[
( \ave{D}_Q, U)_Q = \lim _{t \to +\infty} (A(t), U)_Q = \lim _{t \to +\infty} \inty{Q^{1/2} A(t)Q^{1/2} :Q^{1/2} UQ^{1/2}  }\geq 0
\]
and thus $\ave{D}_Q \geq 0$. By construction $\ave{D}_Q = \mathrm{Proj}_{\ker L} D \in \ker L$. It remains to justify the second statement in \eqref{Equ72}, and \eqref{Equ72Bis}. Take a bounded function $\varphi \in \liy{}$ which remains constant along the flow of $b$, that is $\varphi _s = \varphi, s \in \R$, and a smooth function $u \in C^1 (\R^m)$ such that $u_s = u, s \in \R$ and
\[
\inty{(\nabla _y u \cdot Q^{-1} \nabla _y u )^2 } < +\infty.
\] 
We introduce the matrix field $U$ given by
\[
U(y) = \varphi (y) Q^{-1} (y) \;\nabla _y u \otimes \nabla _y u \; Q^{-1}(y),\;\;y \in \R^m.
\]
By one hand notice that $U \in H_Q$
\begin{align*}
|U|^2_Q & = \inty{Q^{1/2}UQ^{1/2}:Q^{1/2}UQ^{1/2}} = \inty{\varphi ^2|Q^{-1/2} \nabla _y u |^4}\\
& \leq \|\varphi \|_{L^\infty} ^2\inty{(\nabla _y u \cdot Q^{-1} \nabla _y u )^2 }.
\end{align*}
By the other hand, we claim that $U \in \ker L$. Indeed, for any $s \in \R$ we have
\[
\nabla _y u = \nabla _y u_s = {^t \dyy}(\nabla _y u )_s
\]
and thus
\begin{align*}
Q_s U_s Q_s & = \varphi _s (\nabla _y u )_s \otimes ( \nabla _y u )_s \\
& = \varphi \;( {^t \partial _y Y ^{-1}}\nabla _y u ) \otimes ( {^t \partial _y Y ^{-1}}\nabla _y u ) \\
& = \varphi \;{^t \partial _y Y ^{-1}}\;\nabla _y u \otimes \nabla _y u \; \partial _y Y ^{-1}\\
& = {^t \partial _y Y ^{-1}} QUQ \partial _y Y ^{-1}.
\end{align*}
Taking into account that $Q_s = {^t \partial _y Y ^{-1}}Q { \partial _y Y ^{-1}}$ we obtain
\[
{^t \partial _y Y ^{-1}} Q { \partial _y Y ^{-1}}U_s {^t \partial _y Y ^{-1}} Q { \partial _y Y ^{-1}} = {^t \partial _y Y ^{-1}} Q U Q { \partial _y Y ^{-1}} 
\]
saying that $U_s (y)= \dyy U(y) {^t \dyy}$. As $\ave{D}_Q = \mathrm{Proj}_{\ker L }D$ one gets
\begin{align*}
0 = (D - \ave{D}_Q, U ) _Q & = \inty{(D - \ave{D}_Q) : QUQ} \\
& = \inty{\varphi (y) (D - \ave{D}_Q) :\nabla _y u \otimes \nabla _y u }\\
& = \inty{\varphi (y) \{ \nabla _y u \cdot D \nabla _y u - \nabla _y u \cdot \ave{D}_Q \nabla _y u \}}.
\end{align*}
In particular, taking $\varphi = 1$ we deduce that $\nabla _y u \cdot \ave{D}_Q \nabla _y u \in \loy{}$ and
\[
\inty{\nabla _y u \cdot \ave{D}_Q \nabla _y u } = \inty{\nabla _y u \cdot D \nabla _y u } = (D, Q^{-1}\;\nabla _y u \otimes \nabla _y u \;Q^{-1} )_Q < +\infty
\]
since $D \in H_Q$, $Q^{-1}\nabla _y u \otimes \nabla _y u Q^{-1} \in H_Q$. Since $\ave{D}_Q \in \ker L$, the function $\nabla _y u \cdot \ave{D}_Q \nabla _y u $ remains constant along the flow of $b$
\[
(\nabla _y u )_s \cdot (\ave{D}_Q)_s (\nabla _y u )_s = (\nabla _y u )_s \cdot \dyy \ave{D}_Q \;{^t \dyy} (\nabla _y u )_s = \nabla _y u \cdot \ave{D}_Q\nabla _y u.
\]
Therefore the function $\nabla _y u \cdot \ave{D}_Q \nabla _y u $ verifies the variational formulation
\begin{equation}
\label{Equ73} \nabla _y u \cdot \ave{D}_Q \nabla _y u \in \loy{},\;\;(\nabla _y u \cdot \ave{D}_Q \nabla _y u)_s = \nabla _y u \cdot \ave{D}_Q \nabla _y u,\;\;s \in \R
\end{equation}
and
\begin{equation}
\label{Equ74}
\inty{\nabla _y u \cdot D \nabla _y u \;\varphi } = \inty{\nabla _y u \cdot \ave{D}_Q \nabla _y u\;\varphi },\;\;\forall \;\varphi \in \liy{},\;\varphi _s = \varphi,\; s \in \R.
\end{equation} 
It is easily seen, thanks to the hypothesis $D \in \liy{}$, that \eqref{Equ73}, \eqref{Equ74} also make sense for functions $u \in \hoy{}$ such that $u _s = u$, $s \in \R$. We obtain 
\[
\nabla _y u \cdot \ave{D}_Q \nabla _y u = \ave{\nabla _y u \cdot D \nabla _y u },\;\;u \in \hoy{},\;\;u_s = u,\;\;s\in \R
\]
where the average operator in the right hand side should be understood in the $\loy{}$ setting cf. Remark \ref{AveLone}. Moreover, if $u, v \in \hoy{} \cap \kerbg{}$ then $\ave{D}_Q ^{1/2} \nabla _y u, \ave{D}_Q ^{1/2} \nabla _y v$ belong to $\lty{}$ implying that $\nabla _y u \cdot \ave{D}_Q \nabla _y v \in \loy{}$. As before we check that $\nabla _y u \cdot \ave{D}_Q \nabla _y v$ remains constant along the flow of $b$ and for any $\varphi \in \liy{}$, $\varphi _s = \varphi, s \in \R$ we can write
\begin{align*}
2 \inty{\nabla _y u \cdot D \nabla _y v \;\varphi } & = \inty{\nabla _y (u + v) \cdot D \nabla _y (u + v) \;\varphi}\\
& - \inty{\nabla _y u \cdot D \nabla _y u \;\varphi} - \inty{\nabla _y v \cdot D \nabla _y v \;\varphi}\\
& = \inty{\nabla _y (u + v) \cdot \ave{D}_Q \nabla _y (u + v) \;\varphi}\\
& - \inty{\nabla _y u \cdot \ave{D}_Q \nabla _y u \;\varphi} - \inty{\nabla _y v \cdot \ave{D}_Q \nabla _y v \;\varphi}\\
& = 2 \inty{\nabla _y u \cdot \ave{D}_Q \nabla _y v \;\varphi }.
\end{align*}
Finally one gets
\[
\nabla _y u \cdot \ave{D}_Q \nabla _y v = \ave{\nabla _y u \cdot D \nabla _y v},\;\;u, v \in \hoy{} \cap \kerbg{}.
\]
Consider now $u \in \hoy{} \cap \kerbg{}$ and $\psi \in C^2 _c (\R^m)$. In order to prove that $\ave{\nabla _y u \cdot \ave{D}_Q \nabla _y ( b \cdny \psi )} = 0$, where the average is understood in the $\loy{}$ setting, we need to check that 
\[
\inty{\varphi (y) \;\nabla _y u \cdot \ave{D}_Q \nabla _y ( b \cdny \psi ) } = 0
\]
for any $\varphi \in \liy{}$, $\varphi _s = \varphi, s \in \R$. Clearly $B(y) := \varphi (y) \ave{D}_Q (y) \in \ker L$ and therefore it is enough to prove that 
\[
\inty{\nabla _y u \cdot B \nabla _y ( b \cdny \psi ) }= 0
\]
for any $B \in \ker L$, which comes by the third statement of Proposition \ref{WMFI}.
\end{proof}
\begin{remark}
\label{Parametrization} Assume that there is $u_0$ satisfying $u_0 (\ysy) = u_0 (y) + s$, $s \in \R, y \in \R^m$. Notice that $u_0$ could be multi-valued function (think to angular coordinates) but its gradient satisfies for a.a. $y \in \R^m$ and $ s \in \R$
\[
\nabla _y u_0 = {^t \dyy } (\nabla _y u_0 )_s
\]
exactly as any function $u$ which remains constant along the flow of $b$. For this reason, the last equality in \eqref{Equ72} holds true for any $u, v \in \hoy{} \cap \kerbg{} \cup \{u_0\}$. In the case when $m-1$ independent prime integrals of $b$ are known {\it i.e.,} $\exists u_1, ..., u_{m-1} \in \hoy{}\cap \kerbg{}$, the average of the matrix field $D$ comes by imposing
\[
\nabla _y u_i \cdot \ave{D}_Q \nabla _y u_j = \ave{\nabla _y u_i \cdot D \nabla _y u_j},\;\;i, j \in \{0,...,m-1\}.
\]
\end{remark}

\section{First order approximation}
\label{FirstOrdApp} 
\noindent
We assume that the fields $D(y), b(y)$ are bounded on $\R^m$
\begin{equation}
\label{Equ26} D \in \liy{},\;\;b \in \liy{}.
\end{equation}
We solve \eqref{Equ1}, \eqref{Equ2} by using variational methods. We consider the Hilbert spaces $V:= \hoy{} \subset H := \lty{}$ (the injection $V \subset H$ being continuous, with dense image) and the bilinear forms $\aeps : V \times V \to \R$ given by
\[
\aeps (u,v) = \inty{D(y) \nabla _y u \cdot \nabla _y v } + \frac{1}{\eps} \inty{(b \cdny u ) \;(b \cdny v)},\;\;u, v \in V.
\] 
Notice that for any $0 < \eps \leq 1$ and $v \in V$ we have
\begin{align*}
\aeps (v, v) + d |v|_H ^2 & \geq \inty{D(y) \nabla _y v \cdot \nabla _y v + (b \cdny v ) \;(b \cdny v) } + d \inty{(v(y))^2} \\
& \geq d \inty{|\nabla _y v |^2} +  d \inty{(v(y))^2} \\
& = d |v|_V ^2
\end{align*}
saying that $\aeps$ is coercive on $V$ with respect to $H$. By Theorems 1,2 pp. 620 \cite{DauLions88} we deduce that for any $\uein \in H$, there is a unique variational solution for \eqref{Equ1}, \eqref{Equ2}, that is 
$\ue \in C_b (\R_+; H) \cap L^2(\R_+;V)$ and
\[
\ue (0) = \uein,\;\;\frac{\md}{\md t } \inty{\ue (t,y) v(y) } + \aeps (\ue (t), v) = 0,\;\;\mbox{in}\;\dpri{},\;\;\forall \; v \in V.
\]
By standard arguments one gets
\begin{pro}
\label{UnifEstim}
The solutions $(\ue)_\eps$ satisfy the estimates
\[
\|\ue \|_{C_b (\R_+; H)} \leq |\uein|_H,\;\;\int _0 ^{+\infty} \!\!\!\!\inty{|\nabla _y \ue |^2}\md t \leq \frac{|\uein |^2 _H}{2d}
\]
and
\[
\|b \cdny \ue \|_{L^2(\R_+; H)} \leq \left ( \frac{\eps}{2(1 - \eps)}\right ) ^{1/2} |\uein |_H,\;\;\eps \in (0,1).
\]
\end{pro}
We are ready to prove the convergence of the family $(\ue )_\eps$, when $\eps \searrow 0$, towards the solution of the heat equation associated to the averaged diffusion matrix field $\ave{D}_Q$.
\begin{proof} (of Theorem \ref{MainResult1}) Based on the uniform estimates in Proposition \ref{UnifEstim}, there is a sequence $(\eps _k)_k$, converging to $0$, such that 
\[
\uek \rightharpoonup u \;\mbox{ weakly } \star \mbox{ in } L^\infty(\R_+; H),\;\;\nabla _y \uek \rightharpoonup \nabla _y u \;\mbox{ weakly in }\;L^2(\R_+;H).
\]
Using the weak formulation of \eqref{Equ1} with test functions $\eta (t) \varphi (y)$, $\eta \in C^1 _c (\R_+), \varphi \in C^1 _c (\R^m)$ yields
\begin{align}
\label{Equ77} - \intty{\eta ^\prime (t) \varphi (y) \uek (t,y) } & - \eta (0) \inty{\varphi  \uekin } + \intty{\eta  \nabla _y \uek \cdot D \nabla _y \varphi } \nonumber  \\ 
& = - \frac{1}{\eps _k} \intty{\eta (t) ( b \cdny \uek) ( b \cdny \varphi )}. 
\end{align}
Multiplying by $\eps _k$ and letting $k \to +\infty$, it is easily seen that
\[
\intty{\eta (b \cdny u ) \;(b \cdny \varphi ) } = 0.
\]
Therefore $u(t,\cdot) \in \ker {\cal T} = \kerbg$, $t \in \R_+$, cf. Proposition \ref{KerRanTau}. Clearly \eqref{Equ77} holds true for any $\varphi \in V$. In particular, for any $\varphi \in V \cap \kerbg{}$ one gets
\begin{align}
\label{Equ78} - \intty{\eta ^\prime \uek \varphi } - \eta (0) \inty{\uekin \varphi } + \intty{\eta \nabla _y \uek \cdot D \nabla _y \varphi } = 0.
\end{align}
Thanks to the average properties we have
\[
\inty{\uekin \varphi } = \inty{\ave{\uekin} \varphi } \to \inty{\uin \varphi}
\]
and thus, letting $k \to +\infty$ in \eqref{Equ78}, leads to
\begin{align}
\label{Equ79} - \intty{\eta ^\prime u \varphi } - \eta (0) \inty{\uin \varphi } + \intty{\eta \nabla _y u \cdot D \nabla _y \varphi } = 0.
\end{align}
Since $u(t, \cdot), \varphi \in V \cap \kerbg{}$ we have cf. Theorem \ref{AveMatDif}
\[
\inty{\nabla _y u \cdot D \nabla _y \varphi } = \inty{\nabla _y u \cdot \ave{D}_Q \nabla _y \varphi}
\]
and \eqref{Equ79} becomes
\begin{align}
\label{Equ80} - \intty{\eta ^\prime u \varphi } - \eta (0) \inty{\uin \varphi } + \intty{\eta \nabla _y u \cdot \ave{D}_Q \nabla _y \varphi } = 0.
\end{align}
But \eqref{Equ80} is still valid for test functions $\varphi = b \cdny \psi$, $\psi \in C^2 _c (\R^m)$ since $u(t,\cdot) \in \kerbg$, $\uin = \mbox{w}-\lime \ave{\uein} \in \kerbg$ and $\ave{D}_Q \in \ker L$
\[
\inty{u(t,y) b \cdny \psi } = 0,\;\;\inty{\uin b \cdny \psi } = 0,\;\;\inty{\nabla _y u \cdot \ave{D}_Q \nabla _y ( b \cdny \psi ) }= 0
\]
cf. Theorem \ref{AveMatDif}. Therefore, for any $v \in V$ one gets
\[
\frac{\md}{\md t} \inty{u (t,y) v(y) } + \inty{\nabla _y u \cdot \ave{D}_Q \nabla _y v } = 0\;\mbox{ in } \dpri{}
\]
with $u(0) = \uin$. By the uniqueness of the solution of \eqref{Equ75}, \eqref{Equ76} we deduce that all the family $(\ue)_\eps$ converges weakly to $u$.
\end{proof}
\begin{remark}
\label{Propagation} Notice that \eqref{Equ75} propagates the constraint $b \cdny u = 0$, if satisfied initially. Indeed, for any $v \in C^1 _c (\R^m)$ we have
\begin{equation}
\label{Equ81}\frac{\md }{\md t } \inty{u (t,y) v (y) } + \inty{ \nabla _y u \cdot \ave{D}_Q \nabla _y v } = 0\;\mbox{ in } \dpri{}.
\end{equation}
Since $\ave{D}_Q \in \ker L$, we know by the second statement of Proposition \ref{WMFI} that 
\begin{equation*}
%\label{Equ82} 
\inty{\nabla _y u_s \cdot \ave{D}_Q \nabla _y v } = \inty{\nabla _y u \cdot \ave{D}_Q \nabla _y v_{-s}}.
\end{equation*}
Replacing $v$ by $v_{-s}$ in \eqref{Equ81} we obtain
\[
\frac{\md }{\md t } \inty{u_s v } + \inty{\nabla _y u_s \cdot \ave{D}_Q \nabla _y v } = 0\;\mbox{ in } \dpri{}
\]
and therefore $u_s$ solves
\[
\partial _t u_s - \divy ( \ave{D}_Q \nabla _y u_s) = 0,\;\;(t, y) \in \R_+ \times \R^m
\]
and $u_s (0,y) = \uin (\ysy) = \uin (y), y \in \R^m$. By the uniqueness of the solution of \eqref{Equ75}, \eqref{Equ76} one gets $u_s = u$ and thus, at any time $t \in \R_+$, $b \cdny u (t,\cdot) = 0$. 
\end{remark}

\section{Second order approximation}
\label{SecOrdApp}
\noindent
For the moment we have determined the model satisfied by the dominant term in the expansion \eqref{Equ6}. We focus now on second order approximation, that is, a model which takes into account the first order correction term $\eps u ^1$. Up to now we have used the equations \eqref{Equ7}, \eqref{Equ8}. Finding a closure for $u + \eps u ^1$ will require one more equation
\begin{equation}
\label{Equ83} \partial _t u^1 - \divy ( D \nabla _y u^1 ) - \divy ( b \otimes b \nabla _y u^2) = 0,\;\;(t, y) \in \R_+ \times \R ^m.
\end{equation}
Let us see, at least formally, how to get a second order approximation for $(\ue )_\eps$, when $\eps $ becomes small. The first order approximation {\it i.e.}, the closure for $u$, has been obtained by averaging \eqref{Equ8} and by taking into account that $u \in \kerbg{}$
\[
\partial _t u = \ave{\divy( D \nabla _y u ) } = \divy ( \ave{D}_Q \nabla _y u ).
\]
Thus $u^1$ satisfies
\begin{equation}
\label{Equ84} \divy ( \ave{D}_Q \nabla _y u ) - \divy ( D \nabla _y u ) - \divy ( b \otimes b \nabla _y u^1) = 0
\end{equation}
from which we expect to express $u^1$, up to a function in $\kerbg{}$, in terms of $u$.
\begin{proof} (of Theorem \ref{Decomposition})
We claim that $\ran L^2 = \ran L $ and thus $\ran L^2 $ is closed as well. Clearly $\ran L^2 \subset \ran L$. Consider now $Z = L(Y)$ for some $Y \in \dom (L)$. But $Y - \mathrm{Proj}_{\ker L} Y \in \ker L  ^\perp = (\ker L^\star ) ^\perp = \overline{\ran L} = \ran L$ and there  is $X \in \dom (L)$ such that $Y - \mathrm{Proj}_{\ker L} Y = L(X)$. Finally $X \in \dom (L^2)$ and
\[
Z = L(Y) = L(Y - \mathrm{Proj} _{\ker L} Y ) = L(L(X)).
\]
By construction we have $D - \ave{D}_Q \in ( \ker L)^\perp = ( \ker L^\star ) ^\perp = \overline{\ran L} = \ran L = \ran L^2$ and thus there is a unique $F \in \dom (L^2) \cap ( \ker L )^\perp $ such that  $D = \ave{D}_Q - L(L(F))$. As $F \in ( \ker L )^\perp$, there is $C \in \dom (L)$ such that $F = L(C)$ implying that ${^t F} = {^t L(C)} = L ({^t C})$. Therefore ${^t F } \in \dom (L^2) \cap ( \ker L )^\perp$ and satisfies the same equation as $F$
\[
L(L({^t F})) = {^t L}(L(F)) = \ave{D}_Q - D.
\]
By the uniqueness we deduce that $F$ is a field of symmetric matrix.
By Proposition \ref{PropOpeL} we know that
\[
- \divy(L(F) \nabla _y ) = [b \cdot \nabla _y, - \divy ( F \nabla _y )]\;\mbox{ in }\; \dpri{}
\]
{\it i.e.,} 
\[
\inty{L(F) \nabla _y u \cdot \nabla _y v } = - \inty{F \nabla _y u \cdot \nabla _y ( b \cdny v ) } - \inty{F \nabla _y ( b \cdny u ) \cdny v }
\]
for any $u, v \in C^2 _c (\R^m)$. Similarly, $E := L(F)$ satisfies
\[
- \divy ( L^2 (F) \nabla _y ) = - \divy (L(E) \nabla _y ) = [b\cdny, - \divy ( E \nabla _y )]\;\mbox{ in }\;\dpri{}
\]
and thus, for any $u, v \in C^3_c (\R^m)$ one gets
\begin{align*}
& \inty{(\ave{D}_Q - D) \nabla _y u \cdny v }  = \inty{L^2(F)\nabla _y u \cdny v } \\
& = - \inty{L(F) \nabla _y u \cdny ( b \cdny v ) }- \inty{L(F) \nabla _y ( b \cdny u ) \cdny v } \\
& = \inty{F \nabla _y u \cdny ( b \cdny ( b \cdny v ))} + \inty{F \nabla _y ( b \cdny u ) \cdny ( b \cdny v)} \\
& + \inty{F \nabla _y ( b \cdny u ) \cdny ( b \cdny v)} + \inty{F \nabla _y ( b \cdny ( b \cdny u )) \cdny v}.
\end{align*}
\end{proof}
\noindent
The matrix fields $F \in \dom (L^2)$ and $E = L(F) \in \dom (L)$ have the following properties.
\begin{pro}
\label{PropOpeF} For any $u, v \in C^1 (\R^m)$ which are constant along the flow of $b$ we have in $\dpri{}$
\[
D \nabla _y u \cdny v - \ave{D}_Q \nabla _y u \cdny v = - b \cdny ( E \nabla _y u \cdny v ) = - \divy ( b \otimes b \nabla _y ( F \nabla _y u \cdny v ))
\]
and
\[
\ave{E \nabla _y u \cdny v } = \ave{ F \nabla _y u \cdny v } = 0.
\]
In particular 
\[
\inty{E \nabla _y u \cdny v } = \inty{\ave{E \nabla _y u \cdny v}}= 0
\]
\[
\inty{F \nabla _y u \cdny v } = \inty{\ave{F \nabla _y u \cdny v}}= 0
\]
saying that $\ave{\divy ( E \nabla _y u ) } = \ave{\divy ( F \nabla _y u )} = 0$ in $\dpri{}$.
\end{pro}
\begin{proof}
Consider $\varphi \in C^1 _c (\R^m)$, $u, v \in C^1 (\R^m)$ such that $u_s = u, v_s = v$, $s \in \R$ and the matrix field $U = \varphi Q^{-1} \nabla _y v \otimes \nabla _y u Q^{-1} \in H_Q$. Actually $U \in \dom (L)$ and, as in the proof of the last statement in Proposition \ref{PropOpeL}, one gets
\begin{align*}
L(U) & = (b \cdny \varphi ) Q^{-1} \nabla _y v \otimes \nabla _y u Q^{-1} + \varphi \;L ( Q^{-1} \nabla _y v \otimes \nabla _y u Q^{-1}) \\
& =  (b \cdny \varphi)  Q^{-1} \nabla _y v \otimes \nabla _y u Q^{-1}
\end{align*}
since $ Q^{-1} \nabla _y v \otimes \nabla _y u Q^{-1} \in \ker (L)$. Multiplying by $U$ the equality $D - \ave{D}_Q = - L(E)$, $E = L(F)$, one gets
\[
\inty{\varphi ( D - \ave{D}_Q)\nabla _y u \cdny v } = - (L(E), U)_Q = (E, L(U))_Q = \inty{(b \cdny \varphi) ( E \nabla _y u \cdny v )}
\]
implying that $D \nabla _y u \cdny v = \ave{D}_Q \nabla _y u \cdny v - b \cdny ( E \nabla _y u \cdny v)$ in $\dpri{}$. Multiplying by $U$ the equality $E = L(F)$ yields
\[
\inty{\varphi E \nabla _y u \cdny v } = (E, U)_Q = (L(F), U)_Q = - (F, L(U))_Q = - \inty{(b \cdny \varphi) F \nabla _y u \cdny v}.
\]
We obtain 
\[
E \nabla _y u \cdny v = b \cdny ( F \nabla _y u \cdny v) \;\mbox{ in }\; \dpri{}
\]
and thus
\[
D \nabla _y u \cdny v - \ave{D}_Q \nabla _y u \cdny v = - b \cdny (E \nabla _y u \cdny v ) = - b \cdny ( b \cdny ( F \nabla _y u \cdny v ))
\]
in $\dpri{}$. Consider now $U = \varphi Q^{-1} \nabla _y v \otimes \nabla _y u Q^{-1}$ with $\varphi \in \kerbg{}$. We know that $L(U) = 0$ and since, by construction $F \in (\ker L )^\perp$, we deduce 
\[
\inty{\varphi F \nabla _y u \cdny v } = (F, U)_Q = 0
\]
saying that $\ave{F \nabla _y u \cdny v} = 0$. Similarly $E = L(F) \in (\ker L )^\perp$ and $\ave{E \nabla _y u \cdny v } = 0$.
\end{proof}
\begin{remark}
\label{ParametrizationBis}
Assume that there is $u_0$ (eventually multi-valued) satisfying $u_0 (\ysy{}) = u_0 (y) + s$, $s \in \R, y \in \R^m$. Its gradient changes along the flow of $b$ exactly as the gradient of any function which is constant along this flow cf. Remark \ref{Parametrization}. We deduce that $Q^{-1} \nabla _y v \otimes \nabla _y u Q^{-1} \in \ker L$ for any $u, v \in \kerbg \cup \{u_0\}$ and therefore the arguments in the proof of Proposition \ref{PropOpeF} still apply when $u, v \in \kerbg{} \cup \{u_0\}$. In the case when $m-1$ independent prime integrals $\{u_1, ..., u_{m-1}\}$ of $b$ are known, the matrix fields $E, F$ come, by imposing for any $i, j \in \{0,1,...,m-1\}$
\[
- b \cdny (E \nabla _y u_i \cdny u _j) = D \nabla _y u_i \cdny u_j - \ave{D \nabla _y u_i \cdny u_j},\;\;\ave{E \nabla _y u_i \cdny u_j} = 0
\]
and
\[
b \cdny (F \nabla _y u_i \cdny u _j) = E \nabla _y u_i \cdny u _j,\;\;\ave{F \nabla _y u_i \cdny u _j } = 0. 
\]
\end{remark}
We indicate now sufficient conditions which guarantee that the range of $L$ is closed.
\begin{pro}
\label{CompleteIntegr} Assume that \eqref{Equ21}, \eqref{Equ22}, \eqref{Equ23} hold true and that there is a matrix field $R(y)$ such that \eqref{Equ90} holds true. Then the range of $L$ is closed. 
\end{pro}
\begin{proof}
Observe that \eqref{Equ90} implies \eqref{Equ56}. Indeed, it is easily seen that $b \cdny R + R \partial _y b = 0$ in $\dpri{}$ is equivalent to $R = R_s \partial _y Y ( s; \cdot)$, $s \in \R$. We deduce that $P = R ^{-1} \;{^t R} ^{-1}$ satisfies
\[
G(s)P = \partial _y Y ^{-1} (s; \cdot) P_s {^t  \partial _y Y ^{-1} (s; \cdot)} =  \partial _y Y ^{-1} (s; \cdot)R_s ^{-1} \;{^t R_s} ^{-1}  \;{^t \partial _y Y ^{-1} (s; \cdot)} = R^{-1} \;{^t R}^{-1} = P
\]
saying that $[b,P] = 0$ in $\dpri{}$. Therefore we can define $L$ as before, on $H_Q$, which coincides in this case with $\{A:RA\;{^t R} \in \lty{}\}$. We claim that $i \circ L = ( b \cdny ) \circ i$ where $i : H_Q \to \lty{}$, $i(A) = R A\; {^t R}$, $A \in H_Q$, which comes immediately from the equalities
\[
(i\circ G(s))A = RG(s)A {^t R} = R \partial _y Y ^{-1}( s; \cdot ) A_s {^t \partial _y Y }^{-1} {^t R} = R_s A_s {^t R_s} = (i(A))_s,\;s\in \R, A\in H_Q.
\]
In particular we have
\[
\ker L = \{A \in H_Q\;:\; i(A) \in \kerbg\}
\]
and
\begin{align*}
(\ker L )^\perp & = \{A \in H_Q\;:\; \inty{i(A) : U }
= 0\;\forall\;U \in \kerbg{}\} \\
& = \{A \in H_Q\;:\; i(A) \in ( \kerbg)^\perp \}.
\end{align*}
For any $A \in (\ker L)^\perp$ we can apply the Poincar\'e inequality \eqref{Equ23} to $i(A) \in (\kerbg)^\perp$ and we obtain
\[
|A|_Q = |i(A)|_{L^2} \leq C_P |b \cdny (i(A))|_{L^2} = C_P |i (L(A))|_{L^2} = C_P |L(A)|_Q.
\]
Therefore $L$ satisfies a Poincar\'e inequality as well, and thus the range of $L$ is closed. 
\end{proof}
\begin{remark}
\label{ClosedRanL}
The hypothesis  $b \cdny R + R \dyb = 0$ in $\dpri{}$ says that the columns of $R^{-1}$ form a family of $m$ independent vector fields in involution with respect to $b$, cf. Proposition \ref{VFI}
\[
R_s ^{-1} (y) = \dyy R ^{-1} (y),\;\;s\in \R,\;\;y \in \R^m.
\]
\end{remark}
\begin{remark}
\label{ExplicitAve} For any $U \in \ker L$, that is $i(U) \in \kerbg{}$, we have
\[
\inty{R ( D - \ave{D}_Q) {^t R } : i(U)} = 0.
\]
As $\ave{D}_Q \in \ker L $, we know that $i(\ave{D}_Q) = R \ave{D}_Q {^t R } \in \kerbg{}$ and thus the matrix field $R \ave{D}_Q {^t R }$ is the average (along the flow of $b$) of the matrix field $RD\;{^t R}$, which allows us to express $\ave{D}_Q$ in terms of $R$ and $D$
\[
R \ave{D}_Q {^t R } = \ave{R D \;{^t R }}.
\]
\end{remark}
From now on we assume that \eqref{Equ90} holds true. Applying the decomposition of  Theorem \ref{Decomposition} with the dominant term $u \in \kerbg$ in the expansion \eqref{Equ6} and any $v \in C^3 _c (\R^m)$ yields
\[
\inty{(D - \ave{D}_Q) \nabla _y u \cdny v } = - \inty{F \nabla _y u \cdny ( b \cdny ( b \cdny v ))}.
\]
From \eqref{Equ84} one gets
\[
\inty{(D - \ave{D}_Q ) \nabla _y u \cdny v } - \inty{u^1 b \cdny ( b \cdny v )} = 0
\]
and thus 
\begin{equation}
\label{CorrSplit}
u^1 = \divy ( F \nabla _y u ) + v^1,\;\;v^1 \in \ker ( b \cdny ( b \cdny )) = \kerbg.
\end{equation}
Notice that $\ave{u^1} = v^1$, since $\ave{\divy ( F \nabla _y u )} = 0$, cf. Proposition \ref{PropOpeF}. The time evolution for $v^1 = \ave{u^1}$ comes by averaging \eqref{Equ83}
\[
\partial _t v ^1 - \ave{\divy ( D \nabla _y v^1)} - \ave{\divy ( D \nabla _y ( \divy ( F \nabla _y u )))} = 0.
\]
As $v^1 \in \kerbg$ we have
\[
- \ave{\divy ( D \nabla _y v^1)} = - \divy ( \ave{D}_Q \nabla _y v^1)
\]
and we can write, with the notation $w^1 = \divy (F \nabla _y u)$
\begin{align}
\label{Equ86} \partial _t \{u + \eps u^1\} - \divy ( \ave{D}_Q \nabla _y \{u + \eps u^1\}) = \eps \partial _t w^1 - \eps \divy ( \ave{D}_Q \nabla _y w^1 ) + \eps \ave{\divy ( D \nabla _y w^1)}. 
\end{align}
But the time derivative of $w^1$ is given by 
\[
\partial _t w^1 = \divy ( F \nabla _y \partial _t u ) = \divy ( F \nabla _y ( \divy ( \ave{D}_Q \nabla _y u )))
\]
which implies
\begin{align*}
\partial _t w^1 - \divy ( \ave{D}_Q\nabla _y w^1) & = \divy ( F \nabla _y ( \divy ( \ave{D}_Q \nabla _y u )))- \divy ( \ave{D}_Q \nabla _y ( \divy ( F \nabla _y u ))) \\
& = - [\divy(\ave{D}_Q \nabla _y ), \divy ( F \nabla _y )]u.
\end{align*}
Up to a second order term, the equation \eqref{Equ86} writes
\begin{align}
\label{Equ102} \partial _t \{u + \eps u^1\} - \divy ( \ave{D}_Q \nabla _y \{u + \eps u ^1\}) & + \eps [\divy(\ave{D}_Q \nabla _y ), \divy ( F \nabla _y )]\{u + \eps u^1\} \nonumber \\
& - \eps \ave{\divy ( D \nabla _y ( \divy ( F \nabla _y u )))} = {\cal O}(\eps ^2).
\end{align}
We claim that for any $u \in \kerbg$ we have
\begin{equation}
\label{Equ87} \ave{\divy ( D \nabla _y ( \divy ( F \nabla _y u)))} = \ave{\divy ( E \nabla _y ( \divy ( E \nabla _y u )))}.
\end{equation}
By Proposition \ref{PropOpeF} we know that $\ave{\divy ( F \nabla _y u )} = 0$. As $L(\ave{D}_Q) = 0$ we have 
\[
[b \cdny, - \divy ( \ave{D}_Q \nabla _y )] = - \divy ( L ( \ave{D}_Q) \nabla _y ) = 0
\]
and thus $\divy ( \ave{D}_Q\nabla _y)$ leaves invariant the subspace of functions which are constant along the flow of $b$. By the symmetry of the operator $\divy ( \ave{D}_Q \nabla _y )$, we deduce that the subspace of zero average functions is also left invariant by $\divy ( \ave{D}_Q \nabla _y )$. Therefore $\ave{\divy ( \ave{D}_Q \nabla _y ( \divy ( F \nabla _y u )))} = 0$ and
\[
\ave{\divy ( D \nabla _y ( \divy ( F \nabla _y u )))} = \ave{\divy ((D - \ave{D}_Q) \nabla _y ( \divy ( F \nabla _y u )))}.
\]
Thanks to Theorem \ref{Decomposition} we have
\begin{align*}
\divy((D - \ave{D}_Q)\nabla _y ) & = [b \cdny, [b \cdny, - \divy ( F \nabla _y )]\;] \\
& = [b \cdny, - \divy (L(F)\nabla _y )]\\
& = [b \cdny, - \divy (E\nabla _y )]
\end{align*}
which implies that 
\begin{align*}
& \ave{\divy ( D \nabla _y ( \divy ( F \nabla _y u )))}  = \ave{\divy ( (D - \ave{D}_Q) \nabla _y ( \divy ( F \nabla _y u )))}  \\
& = \ave{\divy ( E \nabla _y ( b \cdny ( \divy ( F \nabla _y u )))) - b \cdny ( \divy ( E \nabla _y ( \divy ( F \nabla _y u ))))}\\
& = \ave{\divy ( E \nabla _y ( b \cdny ( \divy ( F \nabla _y u ))))}.
\end{align*}
Finally notice that 
\[
- \divy ( E \nabla _y u ) = - \divy ( L(F)\nabla _y u ) = [b \cdny, - \divy ( F \nabla _y u )] = - b \cdny ( \divy ( F \nabla _y u ))
\]
and \eqref{Equ87} follows. 
We need to average the differential operator $\divy ( E \nabla _y ( \divy ( E \nabla _y )))$ on functions $u \in \kerbg$. For simplicity we perform these computations at a formal level, assuming that all fields are smooth enough. The idea is to express the above differential operator in terms of the derivations ${^t R }^{-1} \nabla _y $ which commute with the average operator (see Proposition \ref{AveComFirstOrder}), since the columns of $R^{-1}$ contain vector fields in involution with $b(y)$. 
\begin{lemma}
\label{ChangeOfCoord} Under the hypothesis \eqref{Equ90}, for any smooth function $u(y)$ and matrix field $E(y)$ we have
\begin{equation}
\label{Equ100} \divy ( E \nabla _y u) = \divy ( R \;{^t E}) \cdot ( {^t R}^{-1} \nabla _y u ) + R E \;{^t R} : ( {^t R } ^{-1} \nabla _y \otimes {^t  R }^{-1} \nabla _y ) u.
\end{equation}
\end{lemma}
\begin{proof}
Applying the formula $\divy (A\xi) = \divy {^t A} \cdot \xi + {^t A } : \partial _y \xi$, where $A(y)$ is a matrix field and $\xi (y)$ is a vector field, one gets
\[
\divy ( E \nabla _y u ) = \divy ( E \;{^t R } \;{^t R ^{-1}} \nabla _y u ) = \divy ( R \;{^t E}) \cdot ( {^t R }^{-1} \nabla _y u ) + R \;{^t E} : \partial _y ( {^t R }^{-1} \nabla _y u ).
\]
The last term in the above formula writes
\begin{align*}
R \;{^t E } : \partial _y ( {^t R }^{-1} \nabla _y u ) & = R \;{^t E} \;{^t R}\; {^t R } ^{-1} : \partial _y ( {^t R } ^{-1} \nabla _y u ) \\
& = R \;{^t E } \;{^t R } : \partial _y ( {^t R} ^{-1} \nabla _y u ) R ^{-1} \\
& = R E \;{^t R} : {^t R }^{-1} \;{^t \partial _y } ( {^t R } ^{-1} \nabla _y u ) \\
& = R E \;{^t R} : ( {^t R} ^{-1} \nabla _y \otimes {^t R} ^{-1} \nabla _y ) u
\end{align*}
and \eqref{Equ100} follows. 
\end{proof}
Next we claim that the term $\ave{\divy ( E \nabla _y ( \divy ( E \nabla _y u )))}$ reduces to a differential operator, if $u \in \kerbg{}$.
\begin{pro}
\label{DifOpe} Under the hypothesis \eqref{Equ90}, for any smooth matrix field $E$ there is a linear differential operator $S(u)$ of order four, such that, for any smooth $u \in \kerbg{}$
\begin{equation}
\label{Equ101} \ave{\divy ( E \nabla _y ( \divy ( E \nabla _y u )))} = S(u).
\end{equation}
\end{pro}
\begin{proof}
For any smooth functions $u, \varphi \in \kerbg{}$ we have, cf. Lemma \ref{ChangeOfCoord}
\begin{align*}
& \inty{\ave{\divy(E \nabla _y ( \divy ( E \nabla _y u )))}\varphi }  = \inty{\divy(E \nabla _y ( \divy ( E \nabla _y u ))) \varphi }\\
& = \inty{\divy ( E \nabla _y u ) \;\divy ( E \nabla _y \varphi )} \\
& = \inty{\{\divy ( R \; ^t E) \cdot ( ^t R ^{-1} \nabla _y u ) + R E \; ^t R : ( ^t R ^{-1} \nabla _y \otimes {^t R } ^{-1} \nabla _y )u   \}\\
& \times \{\divy ( R \; ^t E) \cdot ( ^t R ^{-1} \nabla _y \varphi ) + R E \; ^t R : ( ^t R ^{-1} \nabla _y \otimes {^t R } ^{-1} \nabla _y )\varphi   \}}\\
& = \inty{[\divy (R \;\;^t E) \otimes \divy ( R \;\;^t E)] : [^t R ^{-1} \nabla _y u \otimes {^t R} ^{-1} \nabla _y \varphi] }\\
& + \inty{[R E \;\;^t R \otimes \divy ( R \;\;^t E)] : [( ^t R ^{-1} \nabla _y  \otimes {^t R }^{-1}\nabla _y )u \otimes {^t R } ^{-1} \nabla _y \varphi] }\\
& + \inty{[\divy( R \;\;^t E) \otimes R E \;\;^t R] : [(^t R ^{-1} \nabla _y u ) \otimes ( ^t R ^{-1} \nabla _y \otimes {^t R}^{-1} \nabla _y ) \varphi]}\\
& + \inty{[R E \;\;^t R \otimes R E \;\;^t R]:[ ( ^t R ^{-1} \nabla _y  \otimes {^t R }^{-1}\nabla _y )u \otimes  ( ^t R ^{-1} \nabla _y  \otimes {^t R }^{-1}\nabla _y )\varphi ]}
\end{align*}
Recall that $^tR ^{-1} \nabla _y $ leaves invariant $\kerbg$ and therefore 
\[
{^t R }^{-1} \nabla _y u \otimes {^t R } ^{-1} \nabla _y \varphi \in \kerbg{}
\]
implying that 
\begin{align*}
& \inty{[\divy (R \;\;^t E) \otimes \divy ( R \;\;^t E)] : [^t R ^{-1} \nabla _y u \otimes {^t R} ^{-1} \nabla _y \varphi] }\\
 = & 
\inty{\ave{\divy (R \;\;^t E) \otimes \divy ( R \;\;^t E)} : [^t R ^{-1} \nabla _y u \otimes {^t R} ^{-1} \nabla _y \varphi] }.
\end{align*}
Similar transformations apply to the other three integrals above, and finally one gets
\begin{align*}
\inty{\ave{\divy(E \nabla _y ( \divy ( E \nabla _y u )))}\varphi }  & = \inty{X : [\nablar u \otimes \nablar \varphi ]} \\
& + \inty{Y : [( \nablar \otimes \nablar )u \otimes \nablar \varphi ]}\\
& + \inty{Z : [\nablar u \otimes  ( \nablar \otimes \nablar ) \varphi] } \\
& + \inty{T : [( \nablar \otimes \nablar )u \otimes ( \nablar \otimes \nablar ) \varphi]}\\
& = I_1 (u, \varphi) + I_2 (u, \varphi) + I_3 (u, \varphi) + I_4 (u, \varphi)
\end{align*}
where $\nablar := {^t R} ^{-1} \nabla _y $ and $X, Y, Z, T$ are tensors of order two, three, three and four respectively
\[
X_{ij} = \ave{\divy ( R \;\;^t E) _i \;\divy(R \;\;^t E)_j},\;\;i,j\in \{1,...,m\}
\]
\[
Y_{ijk} = \ave{(R E \;\;^t R) _{ij} \;\divy (R \;\;^t E)_k},\;\;Z_{ijk} = \ave{\divy ( R \;\;^t E)_i \;\;(RE \;\;^t R)_{jk}} ,\;\;i,j, k\in \{1,...,m\}
\]
\[
T_{ijkl} = \ave{(RE \;\;^t R)_{ij} \;\;(RE \;\;^tR)_{kl}},\;\;i,j, k, l\in \{1,...,m\}.
\]
Integrating by parts one gets
\[
I_1 (u, \varphi) = \inty{X \nablar u \cdot \nablar \varphi } = \inty{R^{-1} X \nablar u \cdot \nabla _y \varphi } = \inty{S_1 (u) \varphi}
\]
where $S_1 (u) = - \divy ( R^{-1} X \nablar u)$. Notice that the differential operator 
\[
\xi \to \divy ( R^{-1} \xi) = \divy (\;^t R ^{-1}) \cdot \xi + {^t R}^{-1} : \partial _y \xi
\]
maps $(\kerbg{})^m$ to $\kerbg{}$, since the columns of $R^{-1}$ contain fields in involution with $b$, and therefore $S_1$ leaves invariant $\kerbg{}$, that is, for any $u \in \kerbg{}$, $\xi = X \nablar u \in (\kerbg{})^m$ and $S_1 (u) = - \divy ( R^{-1} X \nablar u ) = - \divy ( R^{-1} \xi ) \in \kerbg{}$. Similarly we obtain
\[
I_2 (u, \varphi) 
%= \inty{Y_{ijk} \nablar _i ( \nablar _j u)\nablar _k \varphi } 
= \inty{S_2 (u) \varphi },\;\;
%S_2 (u) = - \divy ( R^{-1} Y_{ijk} \nablar _i ( \nablar _j u ) e_k)
I_3 (u, \varphi) = \inty{S_3 (u) \varphi },\;\;I_4 (u, \varphi) = \inty{S_4 (u) \varphi }
% \inty{Z_{ijk} \nablar _i u \nablar _j ( \nablar _k \varphi )} = 
\]
where $S_2, S_3, S_4$ are differential operators of order three, three and four respectively, which leave invariant $\kerbg{}$. We deduce that 
\[
\inty{\ave{\divy (E \nabla _y ( \divy ( E \nabla_y u )))} \varphi } = \inty{S(u) \varphi}
\]
for any $u, \varphi \in \kerbg{}$, with $S = S_1 + S_2 + S_3 + S_4$, saying that 
\[
\ave{\divy (E \nabla _y ( \divy ( E \nabla_y u )))} - S(u) \perp \kerbg{}.
\]
But we also know that 
\[
\ave{\divy (E \nabla _y ( \divy ( E \nabla_y u )))} - S(u) \in \kerbg{}
\]
and thus \eqref{Equ101} holds true.
\end{proof}
Combining \eqref{Equ102}, \eqref{Equ87}, \eqref{Equ101} we obtain
\begin{align*}
\partial _t \{u + \eps u^1\} - \divy ( \ave{D}_Q \nabla _y \{u + \eps u ^1\}) & + \eps [\divy(\ave{D}_Q \nabla _y ), \divy ( F \nabla _y )]\{u + \eps u^1\}  \\
& - \eps S(u + \eps u^1)  = {\cal O}(\eps ^2)
\end{align*}
which justifies the equation introduced in \eqref{IntroEqu87}. The initial condition comes formally by averaging the Ansatz \eqref{Equ6}
\[
\ave{\ue} = u + \eps v^1 + {\cal O}(\eps ^2).
\]
One gets
\[
v ^1 (0, \cdot) = \mbox{w-} \lime \frac{\ave{\uein} - \uin }{\eps} = \vin
\]
implying that $u ^1 (0, \cdot) = \vin + \divy(F \nabla _y \uin )$, cf. \eqref{CorrSplit}, which justifies \eqref{NewIC}.

\section{An example}

Let us consider the vector field $b(y) = {^\perp y} := (y_2, - y_1)$, for any $y = (y_1, y_2) \in \R^2$ and the matrix field 
\[
D (y) = \left(
\begin{array}{cc}
\lambda _ 1 (y)  &   0    \\
0  &   \lambda _2 (y)    
\end{array}
\right),\;\;y \in \R^2
\]
where $\lambda _1, \lambda _2 $ are given functions, satisfying $\min _{y\in \R^2} \{\lambda _1 (y), \lambda _2 (y)\} \geq d>0$. We intend to determine the first order approximation, when $\eps \searrow 0$, for the heat equation
\begin{equation}
\label{Equ91} \partial _t \ue - \divy ( D(y) \nabla _y \ue ) - \frac{1}{\eps} \divy ( b(y) \otimes b(y) \nabla _y \ue ) = 0,\;\;(t, y ) \in \R_+ \times \R ^2
\end{equation}
with the initial condition 
\[
\ue (0, y) = \uin (y),\;\;y \in \R^2.
\]
The flow of $b$ is given by $Y(s;y) = {\cal R}(-s)y$, $s \in \R, y \in \R^2$ where ${\cal R}(\alpha)$ stands for the rotation of angle $\alpha \in \R$. The functions in $\kerbg{}$ are those depending only on $|y|$. Notice that the matrix field
\[
R(y) = \frac{1}{|y|}\left(
\begin{array}{rr}
y_2  &   -y_1   \\
y_1  &   y_2     
\end{array}
\right)
\]
satisfies $b \cdot \nabla _y R + R \partial _y b = 0$ and $Q = {^t R } R = I_2$. The averaged matrix field $\ave{D}_Q$ comes, thanks to Remark \ref{ExplicitAve}, by the formula $R \ave{D}_Q {^t R} = \ave{R D \;{^t R}}$ and thus
\[
\ave{D}_Q = {^t R} \ave{RD\; {^t R}} R,\;\;\ave{RD\; {^t R}} = \left(
\begin{array}{rr}
\ave{\frac{\lambda _1 y _2 ^2 + \lambda _2 y_1 ^2 }{|y|^2}}  & \ave{\frac{(\lambda _1 - \lambda _2)y_1 y _2 }{|y|^2}}     \\
 \ave{\frac{(\lambda _1 - \lambda _2)y_1 y _2 }{|y|^2}} &     \ave{\frac{\lambda _1 y _1 ^2 + \lambda _2 y_2 ^2 }{|y|^2}}
\end{array}
\right).
\]
In the case when $\lambda _1, \lambda _2$ are left invariant by the flow of $b$, that is $\lambda _1, \lambda _2$ depend only on $|y|$, it is easily seen that 
\[
\ave{\frac{y_1 ^2}{|y|^2}} = \ave{\frac{y_2 ^2}{|y|^2}} = \frac{1}{2},\;\;\ave{\frac{y_1 y_2}{|y|^2}} = 0
\]
and thus 
\[
\ave{D}_Q = {^t R } \frac{\lambda _1 + \lambda _2}{2} I_2 R = \frac{\lambda _1 + \lambda _2}{2} I_2.
\]
The first order approximation of \eqref{Equ91} is given by
\[
\left\{
\begin{array}{ll}
 \partial _t u - \divy \left ( \frac{\lambda _1 (y) + \lambda _2 (y)}{2} \nabla _y u \right ) = 0,& \;\;(t, y ) \in \R_+ \times \R ^2   \\
 u(0,y) = \uin (y),& \;\;y \in \R^2. 
\end{array}
\right.
\]
We consider the multi-valued function $u_0 (y) = - \theta (y)$, where $y = |y| ( \cos \theta (y), \sin \theta (y))$, which satisfies $b \cdot \nabla _y u_0 = 1$, or $u_0 (Y(s;y)) = u_0 (y) + s$. Notice that the averaged matrix field $\ave{D}_Q$ satisfies (with $u_1 (y) = |y|^2 /2 \in \kerbg{}$\;)
\[
\nabla _y u_i \cdot \ave{D}_Q \nabla _y u _j = \ave{\nabla _y u _i \cdot D \nabla _y u_j },\;\;i, j \in \{0,1\}
\]
as predicted by Remark \ref{Parametrization}.

%%%%%%%%%%%%%%%%%%%%%%%%%%%%%%
%%%%%%%%% APPENDIX %%%%%%%%%%%
%%%%%%%%%%%%%%%%%%%%%%%%%%%%%%

\appendix
\section{Proofs of Propositions \ref{VFI}, \ref{WVFI}, \ref{MFI}, \ref{WMFI}}
\label{A}
\begin{proof} (of Proposition \ref{VFI}) For simplicity we assume that $b$ is divergence free. The general case follows similarly. Let $c(y)$ be a vector field satisfying \eqref{Equ34}. For any vector field $\phi \in C^1 _c (\R^m)$ we have, with the notation $u _\tau = u (Y(\tau;\cdot))$
\begin{align*}
\inty{c \cdot ( \phi _{-h} - \phi )} = \inty{(c_h - c) \cdot \phi } = \inty{(\partial _y Y (h;y) - I) c \cdot \phi }.
\end{align*}
Multiplying by $h^{-1}$ and passing to the limit when $h \to 0$ imply
\[
- \inty{c ( b \cdny \phi ) } = \inty{\partial _y b c\cdot \phi }
\]
and therefore $(b \cdny ) c - \partial _y b c = 0$ in $\dpri{}$.

Conversely, assume that $[b,c] = 0$ in $\dpri{}$. We introduce $e(s,y) = c(Y(s;y)) - \partial _y Y(s;y) c(y)$. Notice that $e(s,\cdot) \in \loloc{}, s \in \R$ and $e(0, \cdot) = 0$. For any vector field $\phi \in C^1 _c (\R^m)$ we have
\[
E_\phi (s) : = \inty{e(s,y) \cdot \phi (y)} = \inty{c(y) \cdot \phi _{-s}} - \inty{\partial _y Y(s;y) c(y) \cdot \phi (y) }
\]
and thus
\begin{align*}
\frac{\md }{\md s} E_\phi (s) & = - \inty{c(y) \cdot ((b \cdny )\phi )_{-s}} - \inty{\partial _y (b(\ysy)) \;c(y) \cdot \phi (y) } \\
& = - \inty{c \cdot (b \cdny ) \phi _{-s}} - \inty{\partial _y b (Y(s;y)) \dyy c(y) \cdot \phi (y) }\\
& = \inty{\dyb \;c(y) \cdot \phi _{-s}} - \inty{\dyb (\ysy) \dyy c(y) \cdot \phi (y)}\\
& = \inty{\dyb (\ysy) ( c(\ysy) - \dyy c(y)) \cdot \phi (y)} \\
& = \inty{e(s,y)\cdot {^t \dyb} (\ysy) \phi (y)}.
\end{align*}
In the previous computation we have used the fact that the derivation and tranlation along $b$ commute
\[
((b \cdny ) \phi )_{-s} = (b\cdny )\phi _{-s}.
\]
After integration with respect to $s$ one gets
\[
E_\phi (s) = \int _0 ^s \inty{e(\tau,y) \cdot {^t \dyb} (Y(\tau;y)) \phi (y) } \;\md \tau.
\]
Clearly, the above equality still holds true for any $\phi \in C_c (\R^m)$. Consider $R>0, T>0$ and let $K = \|{^t \dyb} \circ Y\|_{L^\infty([-T,T] \times B_R)}$. Therefore, for any $s \in [-T, T]$ we obtain
\begin{align*}
\|e(s,\cdot)\|_{L^\infty(B_R)} & = \sup \{ |E_\phi (s)|\;:\;\phi \in C_c (B_R),\;\;\|\phi \|_{\loy{}} \leq 1\}\\
& \leq K \left |\int _0 ^s \|e (\tau, \cdot) \|_{L^\infty (B_R)}\md \tau   \right |.
\end{align*}
By Gronwall lemma we deduce that $\|e(s,\cdot)\|_{L^\infty(B_R)} = 0$ for $-T \leq s \leq T$ saying that $c(\ysy) - \dyy c(y) = 0, s \in \R, y \in \R^m$.
\end{proof}
\begin{proof} (of Proposition \ref{WVFI})\\
1.$\implies$ 2. By Proposition \ref{VFI} we deduce that $c(\ysy) = \dyy c(y)$ and therefore
\begin{align*}
\inty{(c\cdny u) v_{-s}} & = \inty{c(\ysy) \cdot (\nabla _y u ) (\ysy) v(y)}\\
& = \inty{c(y) \cdot {^t \dyy} (\nabla _y u )(\ysy) v(y) } = \inty{(c(y) \cdot \nabla _y u_s)  v(y)}.
\end{align*}
2.$\implies$ 3. Taking the derivative with respect to $s$ of \eqref{Equ41} at $s = 0$, we obtain \eqref{Equ42}.
3.$\implies$ 1. Applying \eqref{Equ42} with $v \in C^1 _c (\R^m)$ and $u _i = y_i \varphi (y)$, $\varphi \in C^2 _c (\R^m)$, $\varphi = 1$ on the support of $v$, yields
\[
\inty{c_i \;b \cdny v } + \inty{c \cdny b_i \;v (y)}= 0
\]
saying that $b \cdny c_i = (\dyb \;c) _i$ in $\dpri{}$, $i \in \{1,...,m\}$ and thus $[b,c] = b \cdny c - \dyb c = 0$ in $\dpri{}$.
\end{proof}
\begin{proof} (of Proposition \ref{MFI}) The arguments  are very similar to those in the proof of Proposition \ref{VFI}. Let us give the main lines. We assume that $b$ is divergence free, for simplicity. Let $A(y)$ be a matrix field satisfying \eqref{Equ35}. For any matrix field $U \in C^1 _c (\R^m)$ we have
\begin{align*}
\inty{A(y) & : ( U(Y(-h;y)) - U(y) )}  = \inty{(A(Y(h;y)) - A(y)) : U(y) } \\
& = \inty{( \partial _y Y (h;y) A(y) {^t  \partial _y Y (h;y)} - A(y)) : U(y)} \\
& = \inty{\{( \partial _y Y (h;y) - I ) A(y) {^t \partial _y Y (h;y)} : U(y) + A(y) {^t (\partial _y Y (h;y) - I)} : U(y)\}}.
\end{align*}
Multiplying by $\frac{1}{h}$ and passing $h \to 0$ we obtain
\[
- \inty{A(y) : ( b \cdny U ) } = \inty{(\dyb A(y) + A(y) {^t \dyb }):U(y) }
\]
saying that $[b,A] = 0$ in $\dpri{}$. 

For the converse implication define, as before
\[
f(s,y) = A(\ysy ) - \dyy A(y) {^t \dyy},\;\;s \in \R,\;\;y \in \R^m.
\]
For any $U \in C^1 _c (\R^m)$ we have
\begin{align*}
F_U (s)& := \inty{f(s,y) : U(y)} \\
&= \inty{A(y) : U(Y(-s;y))} - \inty{\dyy A(y) {^t \dyy } : U (y)}
\end{align*}
and thus
\begin{align*}
\frac{\md }{\md s} F_U (s) & = - \inty{A(y) : (\;(b \cdny )U\;)_{-s}} - \inty{\partial _y ( b (\ysy)) A(y) {^t \dyy } : U(y)} \\
& - \inty{\dyy A(y) {^t \partial _y ( b (\ysy ))} : U(y)} \\
& = - \inty{A(y) : (b \cdny ) U_{-s} } - \inty{\dyb (\ysy) \dyy A(y) {^t \dyy} : U}\\
& - \inty{\dyy A(y) {^t \dyy} {^t \dyb (\ysy) } : U(y)}\\
& = \inty{ \{ \dyb (\ysy) f(s,y) + f(s,y) {^t \dyb (\ysy)}\} : U(y)} \\
& = \inty{f(s,y) : \{ {^t \dyb (\ysy)} U(y) + U(y) \dyb (\ysy) \}}.  
\end{align*}
The previous equality still holds true for $U \in C_c (\R^m)$, and our conclusion follows as in the proof of Proposition \ref{VFI}, by Gronwall lemma. 
\end{proof}
\begin{proof} (of Proposition \ref{WMFI})\\
$1.\implies 2.$ By Proposition \ref{MFI} we deduce that $A(\ysy) = \dyy A(y) {^t \dyy}$. Using the change of variable $y \to \ysy$ one gets
\begin{align*}
\inty{A(y) \nabla _y u \cdot \nabla _y v } & = \inty{A(\ysy) (\nabla _y u )(\ysy) \cdot (\nabla _y v ) (\ysy)} \\
& = \inty{A(y) {^t \dyy }(\nabla _y u ) (\ysy) \cdot {^t \dyy } (\nabla _y v ) (\ysy) } \\
& = \inty{A(y) \nabla _y u_s \cdot \nabla _y v_s }.
\end{align*}
$2.\implies 3.$ Taking the derivative with respect to $s$ at $s = 0$ of the constant function $s \to \inty{A(y) \nabla _y u_s \cdot \nabla _y v_s}$ yields
\[
\inty{A(y) \nabla _y ( b \cdny u ) \cdot \nabla _y v } + \inty{A(y) \nabla _y u \cdot \nabla _y ( b \cdny v) } = 0.
\]
$3.\implies 2.$ For any $u, v \in C^2 _c (\R^m)$ we can write, thanks to 3. applied with the functions $u_s, v_s$
\begin{align*}
\frac{\md }{\md s} \inty{A(y) \nabla _y u_s \cdot \nabla _y v_s} & = \inty{A(y) \nabla _y ( \;( b \cdny u)_s )\cdot \nabla _y v_s} \\
& + \inty{A(y) \nabla _y u_s \cdot \nabla _y ( \; (b \cdny v )_s)} \\
& = \inty{A(y) \nabla _y ( b \cdny u_s) \cdot \nabla _y v_s } \\
& + \inty{A(y) \nabla _y u_s \cdot \nabla _y ( b \cdny v_s ) } = 0.
\end{align*}
Therefore the function $s \to \inty{A(y) \nabla _y u_s \cdot \nabla _y v_s}$ is constant on $\R$ and thus
\[
\inty{A(y) \nabla _y u_s \cdot \nabla _y v_s} = \inty{A(y) \nabla _y u \cdot \nabla _y v},\;\;s\in \R.
\]
Up to now, the symmetry of the matrix $A(y)$ did not play any role. We only need it for the implication $2.\implies 1.$\\
$2.\implies 1.$ We have
\begin{align*}
\inty{A(y) \nabla _y u \cdot \nabla _y v } & = \inty{A(y) \nabla _y u_s \cdot \nabla _y v_s } \\
& = \inty{A(y) {^t \dyy} ( \nabla _y u)_s \cdot {^t \dyy } (\nabla _y v )_s } \\
& = \inty{\dyy A(y) {^t \dyy } (\nabla _y u )_s \cdot ( \nabla _y v )_s } \\
& = \inty{(\partial _y  Y A \;{^t \partial _y  Y})_{-s} \nabla _y u \cdot \nabla _y v }
\end{align*}
where $(\partial _y  Y A {^t \partial _y  Y})_{-s} = \partial _y Y (s; Y(-s;y)) A(Y(-s;y)) {^t \partial _y Y (s; Y(-s;y))}$. We deduce that 
\[
\inty{(A(y) - (\partial _y Y A \;{^t \partial _y Y })_{-s}) \nabla _y u \cdot \nabla _y v} = 0,\;\;u, v \in C^1 _c (\R^m).
\]
Since $A(y) - (\partial _y Y A \;{^t \partial _y Y })_{-s}$ is symmetric, it is easily seen, cf. Lemma \ref{Divergence} below, that $A(y) - (\partial _y Y A \;{^t \partial _y Y })_{-s}= 0$. Therefore we have $A(\ysy) = \dyy A(y) {^t \dyy}$, $s \in \R,y \in \R^m$ and by Proposition \ref{MFI} we deduce that $[b,A] = 0$ in $\dpri{}$. 
\end{proof}
\begin{lemma}
\label{Divergence}
Consider a field $A(y) \in \loloc{}$ of symmetric matrix satisfying 
\begin{equation}
\label{Equ38} \inty{A(y) \nabla _y u \cdot \nabla _y v} = 0,\;\;u, v \in C^1 _c (\R^m).
\end{equation}
Therefore $A(y) = 0$ a.a. $y \in \R^m$.
\end{lemma}
\begin{proof}
Applying \eqref{Equ38} with $v_j = y_j v$, $v \in C^1 _c (\R^m)$, $u_i = y_i \varphi (y)$ where $\varphi \in C^1 _c (\R^m)$ and $\varphi = 1$ on the support of $v$, yields
\begin{equation}
\label{Equ39}
\inty{A(y) e_i \cdot ( y_j \nabla _y v + v e_j )} = 0.
\end{equation}
Applying \eqref{Equ38} with $v$ and $u_{ij} = y_i y_j \varphi (y)$ one gets
\begin{equation}
\label{Equ40}
\inty{A(y) ( y_j e_i + y_i e _j ) \cdot \nabla _y v } = 0.
\end{equation}
Combining \eqref{Equ39}, \eqref{Equ40} we obtin for any $i, j \in \{1,...,m\}$
\[
2 \inty{(A(y)e_i \cdot e_j) \;v(y)} = \inty{( A(y) e_i \cdot e_j + A(y) e_j \cdot e_i)v(y)} = 0 
\]
saying that $A(y) = 0$, a.a. $y \in \R^m$.
\end{proof}

\end{document}